\documentclass{article}
\usepackage[top=2.54cm,bottom=2.0cm,left=2.0cm,right=2.54cm, includeheadfoot]{geometry}

\usepackage[T1]{fontenc}
\usepackage[utf8]{inputenc}
\usepackage[english]{babel}
\usepackage{enumerate}
\usepackage{multirow,booktabs}
\usepackage[table]{xcolor}
\usepackage{fullpage}
\usepackage{lastpage}
\usepackage{indentfirst}
\usepackage{bbm}

\usepackage{footnote}

\usepackage{verbatim}

\usepackage{amsmath,amsfonts,amssymb,amscd,amsthm}

\usepackage{bm}

\usepackage[all,2cell]{xy} \UseAllTwocells \SilentMatrices

\usepackage{hyperref}

\usepackage{subfiles}

\usepackage{multicol}

\newtheorem{teo}{Theorem}[section]
\newtheorem{lem}[teo]{Lemma} 
\newtheorem{cor}[teo]{Corollary}
\newtheorem{prop}[teo]{Proposition} 

\newtheorem{defn}[teo]{Definition} 
\newtheorem{ex}[teo]{Example}

\newtheorem*{claim*}{Claim}

\newtheorem{rem}[teo]{Remark}

\begin{document}

\title{Inductive graded rings, hyperfields and quadratic forms}

\author{Kaique Matias de Andrade Roberto\thanks{Institute of Mathematics and Statistics, University of São Paulo, Brazil.  Emails:   kaique.roberto@usp.br, hugomar@ime.usp.br} \ Hugo Luiz Mariano\thanks{The authors wants to express their gratitude to
Coordena\c c \~ao de Aperfei\c coamento de Pessoal de N\' ivel Superior (Capes -Brazil) by the financial support to develop this work. The second author: program Capes-Print  number 88887.694866/2022-00.}}	

\date{}
\maketitle

\begin{abstract}
 The goal of this work is twofold: (i) to provide a detailed analysis of some categories of inductive graded ring - a concept introduced in \cite{dickmann1998quadratic} in order to provide a solution of Marshall's signature conjecture in the algebraic theory of quadratic forms; (ii) apply this analysis to deepen the connections between the category of special hyperfields (\cite{roberto2021ktheory}) - equivalent to the category of special groups (\cite{dickmann2000special})  and the categories of inductive graded rings.
\end{abstract}

\section*{Introduction}

It can be said that the Algebraic Theory of Quadratic Forms (ATQF) was founded in 1937 by E. Witt, with the introduction of the concept of the Witt ring of a given field, constructed from the quadratic forms with coefficients in the field: given $F$, an arbitrary field of characteristic $\neq 2$, $W(F)$, the Witt ring of $F$, classifies the  quadratic forms over $F$ that are regular and anisotropic, being in one-to-one correspondence with them; thus the focus of the theory is the quadratic forms defined on the ground field where all their coefficients are invertible. In this way, the set of orders in $F$ is in one-to-one correspondence with the set of
minimal prime ideals of the Witt ring of $F$, and more, the set of orders in $F$ provided with the Harrison's topology is a Boolean topological space that, by the bijection above, is identified with a subspace of the Zariski spectrum of the Witt ring of $F$.

Questions about the structure of Witt rings $W(F)$ could only be solved about three decades  after Witt's original idea, through the introduction and analysis of  concept of Pfister form.
The Pfister forms of degree $n \in \mathbb{N}$, in turn, are generators of the power $I^n(F)$ of the fundamental ideal $I(F)\subseteq W(F)$ (the ideal determined by the anisotropic forms of even dimension).

Other finer questions about the powers of the fundamental ideal arose in the early 1970s: J. Milnor, in a seminal article from 1970 (\cite{milnor1970algebraick}), determines a graduated ring $k_\ast(F)$ (from K-theory, reduced mod 2) associated with the field $F$, which interpolates, through graded ring morphisms
$$h_{\ast}(F) : k_{\ast}(F)\longrightarrow\ H^{\ast}(F)\mbox{ and }s_\ast(F) :
k_{\ast}(F)\longrightarrow W_{\ast}(F),$$
the graded Witt ring
$$W_\ast(F) := \bigoplus_{n \in\mathbb{N}} I^n(F)/I^{n+1}(F)$$
and the graded cohomology ring
$$H^\ast(F) := \bigoplus_{n \in \mathbb{N}} H^n(Gal(F^s|F), \{\pm 1\}).$$

From Voevodski's proof of Milnor's conjectures, and the development of special groups theory
(SG) -- an abstract (and first-order) theory of ATQF, introduced by M. Dickmann, and developed by him in partnership with F. Miraglia
since the 1990s -- it has been possible to demonstrate conjectures about signatures put by M. Marshall by T. Lam in the mid-1970s (\cite{dickmann2000special}, \cite{dickmann1998quadratic},\cite{dickmann2003lam}).

The SG theory, which faithfully codifies both the classical theory of quadratic forms over fields and the reduced theory of quadratic forms developed from the 1980s  (\cite{lam1983orderings}), allows us to naturally extend the construction of graded ring functors to all the special  groups $G$: $W(G)$, $W_*(G)$, $k_*(G)$ (\cite{dickmann2000special}, \cite{dickmann2006algebraic}).

The key points in the demonstration of these conjectures for (pre-ordered) fields was a combination of methods: (i) the introduction of Boolean methods in the theory of quadratic forms through the SG theory -especially the Boolean hull functor (\cite{dickmann2000special}, \cite{dickmann2003elementary}); (ii) the encoding of the original problems posed on signatures in questions on graded Witt rings; (iii) the use of Milnor's isomorphisms to transpose these questions to the graded ring of k-theory and the graded ring of cohomology; (iv) the use of Galois cohomology methods to finalize the resolution of the encoded problem.

In \cite{roberto2021ktheory} we developed a k-theory for the category of hyperbolic hyperfields (a category that contains a copy of the category of (pre)special groups): this construction extends, simultaneously, Milnor's k-theory (\cite{milnor1970algebraick}) and Dickmann-Miraglia's k-theory (\cite{dickmann2006algebraic}). An abstract environment that encapsulate all them, and of course, provide an axiomatic approach to guide new extensions of the concept of K-theory in the context of the algebraic and abstract theories of quadratic forms is given by the concept of inductive graded rings  a concept introduced in \cite{dickmann1998quadratic} in order to provide a solution of Marshall's signature conjecture in realm the algebraic theory of quadratic forms for Pythagorean fields.

 The goal of this work is twofold: (i) to provide a detailed analysis of some categories of inductive graded ring; (ii) apply this analysis to deepen the connections between the category of special hyperfields (\cite{roberto2021ktheory}) - equivalent t groups (\cite{dickmann2000special})  and the categories of inductive graded rings.

{\bf Outline of the work:}...

We assume that the reader is familiar with some categorical results concerning adjunctions: mostly are based on \cite{borceux1994handbook1}, but the reader could also consult \cite{mac2013categories}.


\section{Preliminaries: special groups, hyperbolic hyperfields and k-theory}

\subsection{Special Groups}

  
  Firstly, we make a brief summary on special groups. Let $A$ be a set and $\equiv$ a binary relation on $A\times A$. We extend $\equiv$ to a binary relation $\equiv_n$ on $A^n$, by induction on $n\ge1$, as follows:
\begin{enumerate}[i -]
\item $\equiv_1$ is the diagonal relation $\Delta_A \subseteq A \times A$.

 \item $\equiv_2=\equiv$.
 \item If $n \geq 3$, $\langle a_1,...,a_n\rangle\equiv_n\langle b_1,...,b_n\rangle$ if and only there are  $x,y,z_3,...,z_n\in A$ such that 
 $$\langle a_1,x\rangle\equiv\langle b_1,y\rangle,\,
 \langle a_2,...,a_n\rangle\equiv_{n-1}\langle x,z_3,...,z_n\rangle\mbox{ and }\langle b_2,...,b_n\rangle\equiv_{n-1}\langle y,z_3,...,z_n\rangle.$$
\end{enumerate}

Whenever clear from the context, we frequently abuse notation and indicate the afore-described extension $\equiv$ by the same symbol.  

\begin{defn}[Special Group, 1.2 of \cite{dickmann2000special}]\label{defn:sg}
 A \textbf{special group} is a tuple $(G,-1,\equiv)$, where $G$ is a group of exponent 2, 
i.e, $g^2=1$ for all $g\in G$; $-1$ is a distinguished element of $G$, and $\equiv\subseteq G\times 
G\times G\times G$ is a relation (the special relation), satisfying the following axioms for all 
$a,b,c,d,x\in G$:
\begin{description}
 \item [SG 0] $\equiv$ is an equivalence relation on $G^2$;
 \item [SG 1] $\langle a,b\rangle\equiv \langle b,a\rangle$;
 \item [SG 2] $\langle a,-a\rangle\equiv\langle1,-1\rangle$;
 \item [SG 3] $\langle a,b\rangle\equiv\langle c,d\rangle\Rightarrow ab=cd$;
 \item [SG 4] $\langle a,b\rangle\equiv\langle c,d\rangle\Rightarrow\langle 
a,-c\rangle\equiv\langle-b,d\rangle$;
 \item [SG 5] 
$\langle a,b\rangle\equiv\langle c,d\rangle\Rightarrow\langle ga,gb\rangle\equiv\langle 
gc,gd\rangle,\,\mbox{for all }g\in G$.
 \item [SG 6 (3-transitivity)] the extension of $\equiv$ for a binary relation on $G^3$ is a 
transitive relation.
\end{description}
\end{defn}

A group of exponent 2, with a distinguished element $-1$, satisfying the axioms SG0-SG3 and SG5 is called a {\bf  proto special group}; a \textbf{pre special group} is a proto special group that also satisfy SG4. Thus a \textbf{special group} is a pre-special group that satisfies SG6 (or, equivalently, for each $n \geq 1$, $\equiv_n$ is an equivalence relation on $G^n$).

A \textbf{$n$-form} (or form of dimension $n\ge1$) is an $n$-tuple of elements of a pre-SG $G$. An element $b\in G$ is \textbf{represented} on $G$ by the form $\varphi=\langle a_1,...,a_n\rangle$, in symbols $b\in D_G(\varphi)$, if there exists $b_2,...,b_n\in G$ such that $\langle b,b_2,...,b_n\rangle\equiv\varphi$. 

A pre-special group (or special group) 
$(G,-1,\equiv)$ is:\\
$\bullet$ \ \textbf{formally real} if $-1 \notin \bigcup_{n \in \mathbb{N}} D_G( n\langle 1 \rangle)$\footnote{Here the notation $n\langle 1 \rangle$ means the form $\langle a_1,...,a_n\rangle$ where $a_j=1$ for all $j=1,...,n$. In other words, $n\langle 1 \rangle$ is the form $\langle 1 ,...,1\rangle$ with $n$ entries equal to 1.} ;\\
$\bullet$ \ \textbf{reduced} if it is formally real and, for each $a \in G$, $a \in D_G(\langle 1, 1 \rangle)$ iff $a =1$.

\begin{defn}[1.1 of \cite{dickmann2000special}]\label{defnmorph}
 A map $\xymatrix{(G,\equiv_G,-1)\ar[r]^f & (H,\equiv_H,-1)}$ between pre-special groups is a \textbf{morphism 
of pre-special groups or PSG-morphism} if $f:G\rightarrow H$ is a homomorphism of groups, $f(-1)=-1$ and for all 
$a,b,c,d\in G$
$$\langle a,b\rangle\equiv_G\langle c,d\rangle\Rightarrow
\langle f(a),f(b)\rangle\equiv_H\langle f(c),f(d)\rangle$$
A \textbf{morphism of special groups or SG-morphism} is a pSG-morphism between the corresponding pre-special groups. $f$ 
will be an isomorphism if is bijective and $f,f^{-1}$ are PSG-morphisms. 
\end{defn}

It can be verified that a special group  $G$ is formally real iff it admits some SG-morphism $f : G \to 2$. The category of special groups (respectively reduced special groups) and their morphisms will be denoted by $\mathcal{SG}$ 
(respectively $\mathcal{RSG}$).
  
¨

\begin{defn}[2.4 \cite{dickmann2006algebraic}]\label{2.4kt}
$ $
 \begin{enumerate}[a -]
  \item A reduced special group is [MC] if for all $n\ge1$ and all forms $\varphi$ over $G$,
  $$\mbox{For all }\sigma\in X_G,\mbox{ if }\sigma(\varphi)\equiv0\,\mbox{mod }2^n\mbox{ then } \varphi\in I^nG.$$
  \item A reduced special group is [SMC] if for all $n\ge1$,  the multiplication by $\lambda(-1)$ is an injection of $k_nG$ in $k_{n+1}G$.
 \end{enumerate}
\end{defn}


    
    

\subsection{Multifields/Hyperfields}

Roughly speaking, a multiring is a ``ring'' with a multivalued addition, a notion introduced in the 1950s by Krasner's works. The notion of 
multiring was joined to the quadratic forms tools by the hands of M. Marshall in the last decade (\cite{marshall2006real}). We gather the basic information about multirings/hyperfields and expand some details that we use in the context of K-theories. For more detailed calculations involving multirings/hyperfields and quadratic forms we indicate to the reader the reference \cite{ribeiro2016functorial} (or even \cite{worytkiewiczwitt2020witt} and \cite{roberto2021quadratic}). Of course, multi-structures is an entire subject of research (which escapes from the "quadratic context"), and in this sense, we indicate the references \cite{pelea2006multialgebras}, \cite{viro2010hyperfields}, \cite{ameri2019superring}.

\begin{defn}[Adapted from Definition 2.1 in \cite{marshall2006real}]\label{defn:multiring}
 A multiring is a sextuple $(R,+,\cdot,-,0,1)$ where $R$ is a non-empty set, $+:R\times R\rightarrow\mathcal P(R)\setminus\{\emptyset\}$, $\cdot:R\times R\rightarrow R$ and $-:R\rightarrow R$ are functions, $0$ and $1$ are elements of $R$ satisfying:
 \begin{enumerate}[i -]
  \item $(R,+,-,0)$ is a commutative multigroup;
  \item $(R,\cdot,1)$ is a monoid;
  \item $a0=0$ for all $a\in R$;
  \item If $c\in a+b$, then $cd\in ad+bd$ and $dc\in da+db$. Or equivalently, $(a+b)d\subseteq ab+bd$ and $d(a+b)\subseteq da+db$.
  \item If the equalities holds, i.e, $(a+b)d=ab+bd$ and $d(a+b)=da+db$, we said that $R$ is a \textbf{hyperring}\index{hyperring}.
 \end{enumerate}
 
 A multiring is commutative if $(R,\cdot,1)$ is a commutative monoid. A zero-divisor of a multiring $R$ is a non-zero element $a\in R$ such that $ab=0$ for another non-zero element $b\in R$. The multiring $R$ is said to be a multidomain if do not have zero divisors, and $R$ will be a multifield if $1\ne0$ and every non-zero element of $R$ has multiplicative inverse.
\end{defn}
 
 \begin{ex}\label{ex:1.3}
$ $
 \begin{enumerate}[a -]
  \item Suppose that $(G,+,0)$ is an abelian group. Defining $a + b = \{a + b\}$ and $r(g)=-g$, 
we have that $(G,+,r,0)$ is an abelian multigroup. In this way, every ring, domain and field is a multiring, 
multidomain and hyperfield, respectively.
  
  \item $Q_2=\{-1,0,1\}$ is hyperfield with the usual product (in $\mathbb Z$) and the multivalued sum defined by 
relations
  $$\begin{cases}
     0+x=x+0=x,\,\mbox{for every }x\in Q_2 \\
     1+1=1,\,(-1)+(-1)=-1 \\
     1+(-1)=(-1)+1=\{-1,0,1\}
    \end{cases}
  $$
  
  \item Let $K=\{0,1\}$ with the usual product and the sum defined by relations $x+0=0+x=x$, $x\in K$ and 
$1+1=\{0,1\}$. This is a hyperfield  called Krasner's hyperfield \cite{jun2015algebraic}.
  \end{enumerate}
\end{ex}

 Now, another example that generalizes $Q_2=\{-1,0,1\}$. Since this is a new one, we will provide the entire verification that it is a 
multiring:

\begin{ex}[Kaleidoscope, Example 2.7 in \cite{ribeiro2016functorial}]\label{kaleid}
 Let $n\in\mathbb{N}$ and define 
 $$X_n=\{-n,...,0,...,n\} \subseteq \mathbb{Z}.$$ 
 We define the \textbf{$n$-kaleidoscope multiring} by 
$(X_n,+,\cdot,-, 0,1)$, where $- : X_n \to X_n$ is restriction of the  opposite map in $\mathbb{Z}$,  $+:X_n\times 
X_n\rightarrow\mathcal{P}(X_n)\setminus\{\emptyset\}$ is given by the rules:
 $$a+b=\begin{cases}
    \{a\},\,\mbox{ if }\,b\ne-a\mbox{ and }|b|\le|a| \\
    \{b\},\,\mbox{ if }\,b\ne-a\mbox{ and }|a|\le|b| \\
    \{-a,...,0,...,a\}\mbox{ if }b=-a
   \end{cases},$$
and $\cdot:X_n\times X_n\rightarrow X_n$ is given by the rules:
 $$a\cdot b=\begin{cases}
    \mbox{sgn}(ab)\max\{|a|,|b|\}\mbox{ if }a,b\ne0 \\
    0\mbox{ if }a=0\mbox{ or }b=0
   \end{cases}.$$
  With the above rules we have that $(X_n,+,\cdot, -, 0,1)$ is a multiring.
\end{ex}

 Now, another example that generalizes $K=\{0,1\}$.

\begin{ex}[H-hyperfield, Example 2.8 in \cite{ribeiro2016functorial}]\label{H-multi}
Let $p\ge1$ be a prime integer and $H_p:=\{0,1,...,p-1\} \subseteq \mathbb{N}$. Now, define the binary multioperation and operation in $H_p$ as
follows:
\begin{align*}
 a+b&=
 \begin{cases}H_p\mbox{ if }a=b,\,a,b\ne0 \\ \{a,b\} \mbox{ if }a\ne b,\,a,b\ne0 \\ \{a\} \mbox{ if }b=0 \\ \{b\}\mbox{ if }a=0 \end{cases} \\
 a\cdot b&=k\mbox{ where }0\le k<p\mbox{ and }k\equiv ab\mbox{ mod p}.
\end{align*}
$(H_p,+,\cdot,-, 0,1)$ is a hyperfield such that for all $a\in H_p$, $-a=a$. In fact, these $H_p$ are a kind of generalization of $K$, in the sense that $H_2=K$.
\end{ex}
 
 There are many natural constructions on the category of multrings as: products, directed inductive limits, quotients by an ideal,  
localizations by multiplicative subsets and quotients by ideals.  Now, we present some constructions that will be used further. For the first one, we need to restrict our category:

\begin{defn}[Definition 3.1 of \cite{roberto2021quadratic}]\label{hiperb}
An \textbf{hyperbolic multiring} is a multiring $R$ such that $1-1=R$. The category of hyperbolic multirings and hyperbolic hyperfields will be denoted by $\mathcal{HMR}$ and $\mathcal{HMF}$ respectively.
\end{defn}

Let $F_1$ and $F_2$ be two hyperbolic hyperfields. We define a new hyperbolic hyperfield $(F_1\times_h F_2,+,-,\cdot,(0,0),(1,1))$ by the following: the underlying set of this structure is 
$$F_1\times_h F_2:=(\dot F_1\times \dot F_2)\cup\{(0,0)\}.$$
For $(a,b),(c,d)\in F_1\times_h F_2$ we define
\begin{align}\label{prodmultiop}
    -(a,b)&=(-a,-b),\nonumber\\
    (a,b)\cdot(c,d)&=(a\cdot c,b\cdot d),\nonumber \\
    (a,b)+(c,d)&=\{(e,f)\in F_1\times F_2:e\in a+c\mbox{ and }f\in b+d\}\cap(F_1\times_h F_2).
\end{align}
In other words, $(a,b)+(c,d)$ is defined in order to avoid elements of $F_1\times F_2$ of type $(x,0),(0,y)$, $x,y\ne0$.

\begin{teo}[Product of Hyperbolic Hyperfields]\label{hfproduct}
Let $F_1,F_2$ be hyperbolic hyperfields and $F_1\times_h F_2$ as above. Then $F_1\times_h F_2$ is a hyperbolic hyperfield and satisfy the Universal Property of product for $F_1$ and $F_2$.
\end{teo}

In order to avoid confusion and mistakes, we denote the binary product in $\mathcal{HMF}$ by $F_1\times_hF_2$. For hyperfields $\{F_i\}_{i\in I}$, we denote the product of this family by
$$\prod^h_{i\in I}F_i,$$
with underlying set defined by
$$\prod^h_{i\in I}F_i:=\left(\prod_{i\in I}\dot F_i\right)\cup\{(0_i)_{i\in I}\}$$
and operations defined by rules similar to the ones defined in \ref{prodmultiop}. If $I=\{1,...n\}$, we denote
$$\prod^h_{i\in I}F_i=\prod^n_{\substack{i=1 \\ [h]}}F_i.$$

\begin{ex}
Note that if $F_1$ (or $F_2$) is not hyperbolic, then $F_1\times_h F_2$ is not a hyperfield in general. Let $F_1$ be a field (considered as a hyperfield), for example $F_1=\mathbb R$ and $F_2$ be another hyperfield. Then if $a,b\in F_1$, we have
$1-1=\{0\}$, so $(1,a)+(-1,b)=\{0\}\times(a-b)$, and $$[\{0\}\times(a-b)]\cap(F_1\times_h F_2)=\emptyset.$$
\end{ex}


 

\begin{prop}[3.13 of \cite{ribeiro2016functorial}]\label{sg.to.mf}
 Let $(G,\equiv,-1)$ be a special group and define $M(G)=G\cup\{0\}$ where $0:=\{G\}$\footnote{Here, 
the choice of the zero element was ad hoc. Indeed, we can define $0:=\{x\}$ for any $x\notin G$.}. Then 
$(M(G),+,-,\cdot,0,1)$ is a hyperfield, where 
\begin{itemize}
   \item $a\cdot b=\begin{cases}0\,\mbox{if }a=0\mbox{ or }b=0 \\ a\cdot 
b\,\mbox{otherwise}\end{cases}$
   \item $-(a)=(-1)\cdot a$
   \item $a+b=\begin{cases}\{b\}\,\mbox{if }a=0 \\ \{a\}\,\mbox{if }b=0\\ M(G)\,\mbox{if 
}a=-b,\,\mbox{and }a\ne0 
\\ 
D_G(a,b)\,\mbox{otherwise}\end{cases}$
  \end{itemize}
\end{prop}

\begin{cor}[3.14-3.19 of \cite{ribeiro2016functorial}]\label{cor:equiv1}
 The correspondence $G\mapsto M(G)$ extends to an equivalence of categories  $M:\mathcal{SG}\rightarrow \mathcal{SMF}$, frm the category of specail groups nto the category of \textbf{special multifields}.
\end{cor}


\begin{defn}[Definition 3.2 of \cite{roberto2021quadratic}]
A \textbf{Dickmann-Miraglia multiring (or DM-multiring for short)} \footnote{The name ``Dickmann-Miraglia'' is given in honor to professors Maximo Dickmann and Francisco Miraglia, the creators of the special group theory.} is a pair $(R,T)$ such that $R$ is a multiring, $T\subseteq R$ is a multiplicative subset of $R\setminus\{0\}$, and $(R,T)$ satisfies the following properties:
\begin{description}
\item [DM0] $R/_mT$ is hyperbolic.
 \item [DM1] If $\overline{a}\ne0$ in $R/_mT$, then $\overline a^2=\overline 1$ in $R/_mT$. In other words, for all $a\in R\setminus\{0\}$, 
there are $r,s\in T$ such that $ar=s$. 
 \item [DM2] For all $a\in R$, $(\overline 1-\overline a)(\overline 1-\overline a)\subseteq(\overline 1-\overline a)$ in 
$R/_mT$.
 \item [DM3] For all $a,b,x,y,z\in R\setminus\{0\}$, if 
 $$\begin{cases}\overline a\in \overline x+\overline b \\ \overline b\in \overline y+\overline z\end{cases}\mbox{ in }R/_mT,$$
 then exist $\overline v\in\overline x+\overline z$ such that $\overline a\in\overline y+\overline 
v$ and $\overline{vb}\in\overline{xy}+\overline{az}$ in $R/_mT$.
\end{description}

If $R$ is a ring, we just say that $(R,T)$ is a DM-ring, or $R$ is a DM-ring. A Dickmann-Miraglia hyperfield (or DM-hyperfield) $F$ is a 
hyperfield such that $(F,\{1\})$ is a DM-multiring (satisfies DM0-DM3). In other words, $F$ is a DM-hyperfield if $F$ is hyperbolic and for all 
$a,b,v,x,y,z\in F^*$,
\begin{enumerate}[i -]
 \item $a^2=1$.
 \item $(1-a)(1-a)\subseteq(1-a)$.
 \item $\mbox{If }\begin{cases}a\in x+b \\ b\in y+z\end{cases}\mbox{ then there exists }v\in x+z\mbox{ such that }a\in y+v\mbox{ and }vb\in xy+az$.
\end{enumerate}
\end{defn}

\begin{teo}[Theorem 3.4 of \cite{roberto2021quadratic}]\label{teopmf}
 Let $(R,T)$ be a DM-multiring and  denote 
 $$Sm(R,T)=(R/_mT).$$ 
 Then $Sm(R)$ is a special hyperfield (thus $Sm(R,T)^\times$ is a special group).
\end{teo}

\begin{teo}[Theorem 3.9 of \cite{roberto2021quadratic}]
 Let $F$ be a hyperfield satisfying DM0-DM2. Then $F$ satisfies DM3 if and only if satisfies SMF4. In other words, $F$ is a DM-hyperfield if and only if it is a special hyperfield.
\end{teo}

In this sense, we define the following category:

\begin{defn}
A \textbf{pre-special hyperfield} is a hyperfield satisfying DM0, DM1 and DM2. In other words, a pre-special hyperfield is a hyperbolic hyperfield $F$ such that for all $a\in\dot F$, $a^2=1$ and $(1-a)(1-a)\subseteq1-a$.

The category of pre-special hyperfields will be denoted by $PSMF$.
\end{defn}

\begin{teo}\label{psgpsmfhell}
 Let $G$ be a pre-special group and consider $(M(G),+,-,0,1)$, with operations defined by
 \begin{multicols}{2}
 \begin{itemize}
   \item $a\cdot b=\begin{cases}0\,\mbox{if }a=0\mbox{ or }b=0 \\ a\cdot b\,\mbox{otherwise}\end{cases}$
   \item $-(a)=(-1)\cdot a$
   \item $a+b=\begin{cases}\{b\}\,\mbox{if }a=0 \\ \{a\}\,\mbox{if }b=0\\ M(G)\,\mbox{if }a=-b,\,\mbox{and }a\ne0 \\ D_G(a,b)\,\mbox{otherwise}\end{cases}$
  \end{itemize} 
 \end{multicols}
 Then $M(G)$ is a pre-special multifield. Conversely, if $F$ is a pre-special multifield then $(\dot F,\equiv_F,-1)$ is a pre-special group, where
 $$\langle a,b\rangle\equiv_F\langle c,d\rangle\mbox{ iff }ab=cd\mbox{ and }a\in c+d.$$
\end{teo}

We finish this section stating the following result established in \cite{roberto2021hauptsatz}

\begin{teo}[Arason-Pfister Hauptsatz]\label{haup}
 Let $F$ be a special hyperfield, then  it holds $AP_F(n)$, for all $n \geq 0$. In more details: for each  $n \geq 0$ and For each $\varphi = \langle a_1,\cdots, a_k \rangle$, a  non-empty ($k\geq 1$), regular ($a_i \in \dot{F}$) and anisotropic form, if  $\varphi\in I^n(F)$, then $\dim(\varphi)\ge2^n$  $\varphi\in I^n(F) $, if $\varphi \neq \emptyset$ is anisotropic, then $\dim_{W,F}(\varphi)\ge2^n$.
\end{teo}

 \subsection{The K-theory for Multifields/Hyperfields}

In this section we describe the notion of K-theory of a hyperfield, introduced in \cite{roberto2021ktheory} by essentially repeating the construction in \cite{milnor1970algebraick}  replacing the word ``field'' by ``hyperfield'' and explore some of this basic properties. Apart from the obvious resemblance, more technical aspects of this new theory can be developed (but with other proofs) in multistructure setting in parallel with classical K-theory.

    \begin{defn}[The K-theory of a Hyperfield]
  For a hyperfield $F$, $K_*F$ is the graded ring
$$K_*F=(K_0F,K_1F,K_2F,...)$$
  defined by the following rules: $K_0F:=\mathbb Z$. $K_1F$ is the multiplicative group $\dot F$ written additively. 
  With this purpose, we fix the canonical ``logarithm'' isomorphism
$$\rho:\dot F\rightarrow K_1F,$$
  where $\rho(ab)=\rho(a)+\rho(b)$. Then $K_nF$ is defined to be the quotient of the tensor algebra
  $$K_1F\otimes K_1F\otimes...\otimes K_1F\,(n \mbox{ times})$$
  by the (homogeneous) ideal generated by all $\rho(a)\otimes \rho(b)$, with $a, b\ne 0$ and $b\in1-a$. 
  \end{defn}
  
  In other words, for each $n\ge2$, 
$$K_nF:=T^n(K_1F)/Q^n(K_1(F)),$$
where
$$T^n(K_1F):=K_1F\otimes_{\mathbb Z} K_1F\otimes_{\mathbb Z}...\otimes_{\mathbb Z} K_1F$$ 
and $Q^n(K_1(F))$ is the subgroup generated by all expressions of type $\rho(a_1)\otimes\rho(a_2)\otimes...\otimes\rho(a_n)$ such that $a_{i+1}\in1-a_i$ for some $i$ with $1\le i \le n-1$.
  
  To avoid carrying the overline symbol, we will adopt all the conventions used in Dickmann-Miraglia's K-theory (\cite{dickmann2006algebraic}). Just as it happens with the previous K-theories, a generic element $\eta\in K_nF$ has the pattern
  $$\eta=\rho(a_1)\otimes\rho(a_2)\otimes...\otimes\rho(a_n)$$
  for some $a_1,...,a_n\in\dot F$, with $a_{i+1}\in1-a_i$ for some $1\le i<  n$. Note that if $F$ is a field, then ``$b\in1-a$'' just means 
$b=1-a$, and the hyperfield and Milnor's K-theory for $F$ coincide.
  
  The very first task, is to extend the basic properties valid in Milnor's and Dickmann-Miraglia's K-theory to ours. Here we already need to 
restrict our attention to {\bf hyperbolic hyperfields} ($\mathcal{HMF}$):

  \begin{lem}[Basic Properties I]\label{bp1}
 Let $F$ be an hyperbolic hyperfield. Then
 \begin{enumerate}[a -]
 \item $\rho(1)=0$.
 \item For all $a\in\dot F$, $\rho(a)\rho(-a)=0$ in $K_2F$.
 \item For all $a,b\in\dot F$, $\rho(a)\rho(b)=-\rho(b)\rho(a)$ in $K_2F$.
 \item  For every $a_1,...,a_n\in\dot F$ and every permutation $\sigma\in S_n$,
 $$\rho(a_{\sigma 1})...\rho(a_{\sigma i})...\rho(  a_{\sigma n})=\mbox{sgn}(\sigma)\rho(a_1)...\rho(a_n)\mbox{ in }K_nF.$$
  \item For every $\xi\in K_mF$ and $\eta\in K_nF$, $\eta\xi=(-1)^{mn}\xi\eta$ in $K_{m+n}F$.
 \item For all $a\in\dot F$, $\rho(a)^2=-\rho(a)\rho(-1)$.
 \end{enumerate}
\end{lem}
\begin{proof}
 $ $
 \begin{enumerate}[a -]
  \item Is an immediate consequence of the fact that $\rho$ is an isomorphism.
  \item Since $F$ hyperbolic, $1-1=F$. Then $-a^{-1}\in1-1$ for all $a\in\dot F$, and hence, $-1\in-1+a^{-1}$. Multiplying this by 
$a$, we get $-a\in1-a$. By definition, this imply $\rho(a)\rho(-a)=0$.
  
  \item By item (b), $\rho(ab)\rho(-ab)=0$ in $K_2F$. But
  \begin{align*}
   \rho(ab)\rho(-ab)&=\rho(a)\rho((-a)b)+\rho(b)\rho((-b)a) \\
   &=\rho(a)\rho(-a)+\rho(a)\rho(b)+\rho(b)\rho(-b)+\rho(b)\rho(a) \\
   &=\rho(a)\rho(b)+\rho(b)\rho(a).
  \end{align*}
  From $\rho(a)\rho(b)+\rho(b)\rho(a)=\rho(ab)\rho(-ab)=0$, we get the desired result $\rho(a)\rho(b)=-\rho(b)\rho(a)$ in $K_2F$.
  
  \item This is a consequence of item (c) and an inductive argument.
  
  \item This is a consequence of item (d) and an inductive argument, using the fact that an element in $K_nF$ has a pattern
  $$\eta=\rho(a_1)\otimes \rho(a_2)\otimes...\otimes \rho(a_n)$$
  for some $a_1,...,a_n\in\dot F$, with $a_{i+1} \in 1-a_i$ for some $1\le i< n$.
 
 \item Direct consequence of item (a). 
 \end{enumerate}
\end{proof} 

An element $a\in\dot F$ induces a morphism of graded rings $\omega^a=\{\omega^a_n\}_{n\ge1}:K_*F\rightarrow K_*F$ of degree 1, where $\omega^a_n:K_nF\rightarrow K_{n+1}F$ is the multiplication by $\rho(a)$. When $a=-1$, we write
$$\omega=\{\omega_n\}_{n\ge1}=\{\omega^{-1}_n\}_{n\ge1}=\omega^{-1}.$$

 \begin{prop}[Adapted from 3.3 of \cite{dickmann2006algebraic}]\label{3.3ktmultiadap}
 Let $F,K$ be hyperbolic hyperfields and $\varphi:F\rightarrow L$ be a morphism. Then $\varphi$ induces a morphism of graded rings
 $$\varphi_*=\{\varphi_n:n\ge0\}:K_*F\rightarrow K_*L,$$
 where $\varphi_0=Id_{\mathbb Z}$ and for all $n\ge1$, $\varphi_n$ is given by the following rule on generators
$$\varphi_n(\rho(a_1)...\rho(a_n))=\rho(\varphi(a_1))...\rho(\varphi(a_n)).$$
Moreover if $\varphi$ is surjective then $\varphi_*$ is also surjective, and if $\psi:L\rightarrow M$ is another morphism then
\begin{enumerate}[a -]
 \item $(\psi\circ\varphi)_*=\psi_*\circ\varphi_*$ and $Id_*=Id$.
 \item For all $a\in \dot F$ the following diagram commute:
 $$\xymatrix@!=5pc{K_nF\ar[d]_{\varphi_n}\ar[r]^{\omega^a_n} & K_{n+1}F\ar[d]^{\varphi_{n+1}} \\ K_nL\ar[r]_{\omega^{\varphi(a)}_n} & K_{n+1}L}$$
 \item 
 For all $n\ge1$ the following diagram commute:
 $$\xymatrix@!=5pc{K_nF\ar[d]_{\varphi_n}\ar[r]^{\omega^{-1}_n} & K_{n+1}F\ar[d]^{\varphi_{n+1}} \\ K_nL\ar[r]_{\omega^{-1}_n} & K_{n+1}L}$$
\end{enumerate}
\end{prop}

  We also have the reduced K-theory graded ring  $k_*F=(k_0F,k_1F,...,k_nF,...)$ in the hyperfield context, which is defined by the rule $k_nF:=K_nF/2K_nF$ for all $n\ge0$. Of course
  for all $n\ge0$ we have an epimorphism $q:K_nF\rightarrow k_nF$ simply denoted by $q(a):=[a]$, $a\in K_nF$. It is immediate that $k_nF$ is additively generated by $\{[\rho(a_1)]..[\rho(a_n)]:a_1,...,a_n\in\dot F\}$. We simply denote such a generator by $\tilde\rho(a_1)...\tilde\rho(a_n)$ or even $\rho(a_1)...\rho(a_n)$ whenever the context allows it.
  
  We also have some basic properties of the reduced K-theory, which proof is just a translation of 2.1 of \cite{dickmann2006algebraic}:
  
  \begin{lem}[Adapted from 2.1 \cite{dickmann2006algebraic}]\label{2.1ktmulti}
 Let $F$ be a hyperbolic hyperfield, $x,y,a_1,...,a_n\in\dot F$ and $\sigma$ be a permutation on $n$ elements.
 \begin{enumerate}[a -]
  \item In $k_2F$, $\rho(a)^2=\rho(a)\rho(-1)$. Hence in $k_mF$, 
$\rho(a)^m=\rho(a)\rho(-1)^{m-1}$, $m\ge2$;
\item In $k_2F$, $\rho(a)\rho(b)=\rho(b)\rho(a)$;
\item In $k_nF$, 
$\rho(a_1)\rho(a_2)...\rho(a_n)=\rho(a_{\sigma 1})\rho(a_{\sigma 
2})...\rho(a_{\sigma n})$;
\item For $n\ge1$ and $\xi\in k_nF$, $\xi^2=\rho(-1)^n\xi$;
\item If $F$ is a real reduced hyperfield, then $x\in1+y$ and $\rho(y)\rho(a_1)...\rho(a_n)=0$ implies
$$\rho(x)\rho(a_1)\rho(a_2)...\rho(a_n)=0.$$
 \end{enumerate}
\end{lem}

  Moreover the results in Proposition \ref{3.3ktmultiadap} continue to hold if we took $\varphi_*=\{\varphi_n:n\ge0\}:k_*F\rightarrow k_*L$.
  
  \begin{prop}\label{ktmarshall1}
  Let $F$ be a (hyperbolic) hyperfield and $T\subseteq F$ be a multiplicative subset such that $F^2\subseteq T$. Then, for each $n \geq 1$ 
  $$K_n(F/_m T^*)\cong k_n(F/_mT^*).$$
  \end{prop}
  \begin{proof}
  Since $F^2\subseteq T$, for all $a\in (F/_mT^*)^\times$ we have
  $$0= \rho(1) = \rho(a^2)=\rho(a)+\rho(a).$$
  Then, for each $n \geq 1$, $2K_n(F/_mT^*)=0$ and we get $K_n(F/_m T^*)\cong k_n(F/_mT^*)$.
  \end{proof} 
  
  \begin{teo}\label{ktmarshall2}
  Let $F$ be a hyperbolic hyperfield and $T\subseteq F$ be a multiplicative subset such that $F^2\subseteq T$. Then there is an induced surjective morphism
  $$k(F)\rightarrow k(F/_mT^*).$$
  Moreover, if $T = F^2$, then
  $$k(F)
  \overset\cong\to k(F/_m\dot F^2).$$
  \end{teo}

  \section{Inductive Graded Rings: An Abstract Approach}

  After the three K-theories defined in the above sections, it is desirable (or, at least, suggestive) the rise of an abstract environment that encapsule all them, and of course, provide an axiomatic approach to guide new extensions of the concept of K-theory in the context of the algebraic and abstract theories of quadratic forms. The inductive graded rings fits this purpose. Here we will present three versions. The first one is:
  
\begin{defn}[Inductive Graded Rings First Version (adapted from Definition 9.7 of \cite{dickmann2000special})]\label{igr1}
 An \textbf{inductive graded ring} (or \textbf{Igr} for short) is a structure $R=((R_n)_{n\ge0},(h_n)_{n\ge0},\ast_{nm})$ where
\begin{enumerate}[i -]
    \item $R_0\cong\mathbb F_2$.
    \item $R_n$ has a group structure $(R_n,+,0,\top_n)$ of exponent 2 with a distinguished element $\top_n$.
    \item $h_n:R_n\rightarrow R_{n+1}$ is a group homomorphism such that $h_n(\top_n)=\top_{n+1}$.
    \item For all $n\ge1$, $h_n=\ast_{1n}(\top_1,\_)$.
    \item The binary operations $\ast_{nm} : R_n \times R_m \to R_{n+m}$, $n, m \in \mathbb{N}$ induces a commutative  ring structure on the abelian group
    $$R=\bigoplus_{n\ge0}R_n$$
    with $1=\top_0$.
    \item For $0\le s\le t$ define
    $$h^t_s=\begin{cases}Id_{R_s}\mbox{ if }s=t\\
    h_{t-1}\circ...\circ h_{s+1}\circ h_s\mbox{ if }s<t.\end{cases}$$
    Then if $p\ge n$ and $q\ge m$, for all $x\in R_n$ and $y\in R_m$,
    $$h^p_n(x)\ast h^q_m(y)=h^{p+q}_{n+m}(x\ast y).$$
\end{enumerate}
 A \textbf{morphism} between Igr's $R$ and $S$ is a pair $f=(f,(f_n)_{n\ge0})$ where $f_n:R_n\rightarrow S_n$ is a morphism of pointed groups and 
$$f=\bigoplus\limits_{n\ge0}f_n:R\rightarrow S$$
is a morphism of commutative rings with unity. The category of inductive graded rings (in first version) and their morphisms will be denoted by $\mbox{Igr}$.
\end{defn}

A first consequence of these definitions is that: if 
$$f:((R_n)_{n\ge0},(h_n)_{n\ge0},\ast_{nm})\rightarrow ((S_n)_{n\ge0},(l_n)_{n\ge0},\ast_{nm})$$ 
is a morphism of Igr's then $f_{n+1}\circ h_n=l_n\circ f_n$.
$$\xymatrix@!=2.5pc{R_0\ar[r]^{h_0}\ar[d]_{f_0} & R_1\ar[r]^{h_1}\ar[d]_{f_1} & R_2\ar[r]^{h_2}\ar[d]_{f_2} & 
...\ar[r]^{h_{n-1}} & R_n\ar[r]^{h_n}\ar[d]_{f_n} & R_{n+1}\ar[r]^{h_{n+1}}\ar[d]_{f_{n+1}} & ... \\  
S_0\ar[r]^{l_0} & S_1\ar[r]^{l_1} & S_2\ar[r]^{l_2} & ...\ar[r]^{l_{n-1}} & S_n\ar[r]^{l_n} & S_{n+1}\ar[r]^{l_{n+1}} & 
...}$$
In fact, since $R_0 \cong \mathbb{F}_2 \cong S_0$ and $f(1) =1$,  then $f_0:R_0\rightarrow S_0$ is the unique abelian group isomorphism and $f_1\circ h_0=l_0\circ f_0$. If $n\ge1$, for all $a_n\in R_n$ holds
\begin{align*}
 f_{n+1}\circ h_n(a_n)&=f_{n+1}\circ(\ast_{1n}(\top_1,a_n))=f_1(\top_1)\ast_{1n}f_n(a_n) \\
 &=\top_1\ast_{1n}f_n(a_n)=l_n(f_n(a_n))=l_n\circ f_n(a_n).
\end{align*}

\begin{ex}\label{ex1}
$ $
 \begin{enumerate}[a -]
  \item  Let $F$ be a field of characteristic not 2. The main actors here are $WF$, the Witt ring of $F$ and $IF$, the fundamental ideal of $WF$. Is well know that $I^nF$, the $n$-th power of $IF$ is additively generated by $n$-fold Pfister forms over $F$. Now, let $R_0=WF/IF\cong\mathbb F_2$ 
and $R_n=I^nF/I^{n+1}F$. Finally, let $h_n=\_\otimes\langle1,1\rangle$. With these prescriptions we have an inductive graded ring $R$ associated to $F$.
  
  \item The previous example still works if we change the Witt ring of a field $F$ for the Witt ring of a (formally real) special group $G$.
 \end{enumerate}
\end{ex}

Concerning k-theories, we register the followings:

\begin{teo}\label{km1}
$ $
 \begin{enumerate}[a -]
  \item Let $F$ be a field. Then $k^{mil}_*F$ (the reduced Milnor K-theory) is an inductive graded ring.
  \item Let $G$ be a special group. Then $k^{dm}_*G$ (the Dickmann-Miraglia K-theory of $G$) is an inductive graded ring.
  \item Let $F$ be a hyperbolic hyperfield. Then $k^{mult}_*F$ (our reduced K-theory) is an inductive graded ring.
 \end{enumerate}
\end{teo}

 \begin{teo}[Theorem 2.5 in \cite{dickmann2003lam}]\label{km4}
   Let $F$ be a field. The functor $G:Field_2\rightarrow SG$ provides a functor $k'^{dm}_*:Field_2\rightarrow \mbox{Igr}$ (the special group 
K-theory functor) given on the objects by $k'^{dm}_*(F):=k^{dm}_*(G(F))$ and on the morphisms $f:F\rightarrow K$ by $k'^{dm}_*(f):=G(f)_*$ (in the sense of Lemma 3.3 of \cite{dickmann2006algebraic}). Moreover, this functor commutes with the functors $G$ and $k$, i.e, for all $F\in Field$, 
$k'^{dm}_*(F) = k^{dm}_*(G(F))\cong k^{mil}_*(F)$.
  \end{teo}

  \begin{teo}\label{km5}
   Let $G$ be a special group. The equivalence of categories $M:SG\rightarrow SMF$ induces a functor $k'^{mult}_*:SG\rightarrow 
\mbox{Igr}$ given on the objects by $k'^{mult}_*(G):=k^{mult}_*(M(G))$ and on the morphisms $f:G\rightarrow H$ by 
$k'^{mult}_*(f):=k^{mult}_*(M(f))$. Moreover, this functor commutes with $M$ and $k^{dm}$, i.e, for all $G\in SG$, $k'^{mult}_*(G)\ = k^{mult}_*(M(G))\cong k^{dm}_*(G)$.
  \end{teo}

   \begin{teo}[Interchanging K-theories Formulas]\label{res0}
   Let $F\in Field_2$. Then
   $$k^{mil}(F)\cong k^{dm}(G(F))\cong k^{mult}(M(G(F))).$$
   If $F$ is formally real and $T$ is a preordering of $F$, then
   $$k^{dm}(G_T(F))\cong k^{mult}(M(G_T(F))).$$
   Moreover, since $M(G(F))\cong F/_m\dot F^2$ and $M(G_T(F))\cong F/_mT^*$, we get
   \begin{align*}
       k^{mil}(F)&\cong k^{dm}(G(F))\cong k^{mult}(F/_m\dot F^2)\mbox{ and } \\
       k^{dm}(G_T(F))&\cong k^{mult}(F/_mT^*).
   \end{align*}
  \end{teo}

There is an alternative definition for $\mbox{Igr}$ with a first-order theoretic flavor. It is a technical framework that allows achieving some model-theoretic results.

Before define it, we need some preparation. First of all, we set up the language. Here, we will work with the poli-sorted framework (as established in chapter 5 of \cite{adamek1994locally}), which means the following:

Let $S$ be a set (of sorts). For each $s\in S$ assume a countable set $\mbox{Var}_s$ of  \textbf{variables of sort $s$} (with the convention if $s\ne t$ then $\mbox{Var}_s\cap\mbox{Var}_t=\emptyset$). For each sort  $s\in S$ an equality symbol $=_s$ (or just $=$); the connectives $\neg, \wedge, \vee, \to$ (not, and, or, implies); the quantifiers $\forall, \exists$ (for all, there exists).

A \textbf{finitary $S$-sorted language (or signature)} is a set $\mathcal L=(\mathcal C,\mathcal F,\mathcal R)$ where:
\begin{enumerate}[i -]
 \item $\mathcal C$ is the set of constant symbols. For each $c\in\mathcal C$ we assign an element $s\in S$, the sort of $c$;
 \item $\mathcal F$ is the set of functional symbols. For each $f\in\mathcal F$ we assign elements $s,s_1,...,s_n\in S$,  we say that $f$ has arity $s_1\times...\times s_n$ and $s$ is the value sort of $f$; and we use the notation $f:s_1\times...\times s_n\rightarrow s$.
 \item $\mathcal R$ is the set of relation symbols. $c\in\mathcal C$ we assign elements $s_1,...,s_n\in S$, the arity of $R$; and we say that $R$ has arity $s_1\times...\times s_n$.
\end{enumerate}

A \textbf{$\mathcal L$-structure} $\mathcal M$ is, in this sense, prescribed by the following data:
\begin{enumerate}[i-]
 \item The \textbf{domain or universe} of $\mathcal M$, which is an $S$-sorted set $|\mathcal M|:=(M_s)_{s\in S}$.
 \item For each constant symbol $c\in\mathcal C$ of arity $s$, an element $c^{\mathcal M}\in M_s$.
 \item For each functional symbol $f\in\mathcal F$, $f:s_1\times...\times s_n\rightarrow s$, a function
 $f^{\mathcal M}:M_{s_1}\times...\times M_{s_n}\rightarrow M_s$.
 \item For each relation symbol $R\in\mathcal R$ of arity $s_1\times...\times s_n$ a relation, i.e. a subset $R^{\mathcal M}\subseteq M_{s_1}\times...\times M_{s_n}$.
\end{enumerate}
 A \textbf{$\mathcal L$-morphism} $\varphi:\mathcal M\rightarrow\mathcal N$ is a sequence of functions $\varphi = (\varphi_s)_s :|\mathcal M|\rightarrow|\mathcal N|$ such that
 \begin{enumerate}[i -]
  \item for all $c\in\mathcal C$ of arity $s$, $\varphi_s(c^{\mathcal M})=c^{\mathcal N}$;
  \item for all $f:s_1\times...\times s_n\rightarrow s$, if  $(a_1,...,a_n)\in :M_{s_1}\times...\times M_{s_n}$, then $\varphi_s(f^{\mathcal M}(a_1,...,a_n))=f^{\mathcal N}(\varphi_{s_1}(a_1),...,\varphi_{s_n}(a_n))$;
  \item for all $R$ of arity $s_1\times...\times s_n$, if $(a_1,...,a_n)\in R^{\mathcal M}$ then $(\varphi(a_1),...,\varphi(a_n))\in R^{\mathcal N}$.
 \end{enumerate}
The category of $\mathcal L$-structures and $\mathcal L$-morphism in the poli-sorted language $\mathcal L$ will be denoted by $\mbox{Str}_s(\mathcal L)$.

The terms, formulas, occurrence and free variables definitions for the poli-sorted case are similar to the usual (single-sorted) first order ones. For example, the terms are defined as follows:
\begin{enumerate}[i -]
 \item variables $x\in\mbox{Var}_s$ and constants $c\in C_s$ are terms of value sort $s$;
 \item if $\vec s=\langle s_1,...,s_n,s\rangle\in S^{n+1}$, $f\in\mathcal F$ with $f:s_1\times...\times s_n\rightarrow s$, and $\tau_1,...,\tau_n$ are terms of value sorts $s_1,...,s_n$ respectively, then $f(\tau_1,...,\tau_n)$ is a term of sort  $s$.
\end{enumerate}
As usual, we may write $\tau : s$ to indicate that the term $\tau$ has value sort $s$.

For the formulas:
\begin{enumerate}[i -]
 \item if $x,y\in\mbox{Var}_s$ then $x=y$ is a formula; if $\vec s=\langle s_1,...,s_n\rangle\in S^n$,  $R\in\mathcal R$ of arity $s_1\times...\times s_n$ and $\tau_1,...,\tau_n$ are terms of sort $s_1,...,s_n$ respectively, then $R(\tau_1,...,\tau_n)$ is a formula. These are the \textbf{atomic formulas}.
 \item If $\varphi_1,\varphi_2$ are formulas, then $\neg\varphi_1$, $\varphi_1\wedge\varphi_2$, $\varphi_1\vee\varphi_2$ and $\varphi_1\to\varphi_2$ are formulas.
 \item If $\varphi$ is a formula and $x\in\mbox{Var}_s$ ($s\in S$), then $\forall x \varphi$ and $\exists x \varphi$ are formulas.
\end{enumerate}

In our particular case, the set of sorts will be just $\mathbb N$. Then, for each $n,m\ge0$, we set the following data:
\begin{enumerate}[i -]
 \item $0_n,\top_n$ are constant symbols of arity $n$. We use $0_0=0$ and $\top_0=1$.
 \item $+_n:n\times n\rightarrow n$ is a binary operation symbol.
 \item  $h_n:n\rightarrow(n+1)$ and $\ast_{n,m}:n\times m\rightarrow(n+m)$ are functional symbols.
\end{enumerate}

The \textbf{(first order) language of inductive graded rings} $\mathcal L_{igr}$ is just the following language (in the poli-sorted sense):
$$\mathcal L_{igr}:=
\{0_n,\top_n,+_n,h_n,\ast_{nm}:n,m\ge0\}.$$

The \textbf{(first order) theory of inductive graded rings} $T(\mathcal L_{igr})$ is the $\mathcal L_{igr}$-theory axiomatized by the following $\mathcal L_{igr}$-sentences, where we use  $\cdot_n:0\times n\rightarrow n$ as an abbreviation for $\ast_{0n}$:
\begin{enumerate}[i -]
 \item For $n\ge0$, sentences saying that ``$+_n,0_n,\top_n$ induces a pointed left $\mathbb F_2$-module'':
 \begin{align*}  &\forall\,x:n\forall\,y:n\forall\,z:n((x+_ny)+_nz=x+_n(y+_nz)) \\
  &\forall\,x:n(x+_n0_n=x) \\
  &\forall\,x:n\forall\,y:n(x+_ny=y+_nx) \\
  &\forall\,x:n(x+_nx=0_n) \\
  &\forall\,x:n(1\cdot_n x=x) \\  &\forall\,x:n\forall\,y:n\forall\,a:0(a\cdot_n(x+_ny)=a\cdot_n x+_na\cdot_n y) \\
  &\forall\,x:n\forall\,a:0\forall\,b:0((a+_0b)\cdot_n x=a\cdot_n x+_nb\cdot_n x)
 \end{align*}
 
 \item For $n\ge0$, sentences saying that ``$h_n$ is a pointed $\mathbb F_2$-morphism'':
 \begin{align*}
  &\forall\,x:n\forall\,y:n(h_n(x+_ny)=h_n(x)+_{n+1}h_n(y)) \\
  &\forall\,x:n\forall\,a:0(h_n(a\cdot_n x)=a\cdot_n h_n(x)) \\
  &h_n(\top_n)=\top_{n+1}
 \end{align*}
 
 \item Sentences saying that ``$R_0\cong\mathbb F_2$'':
 \begin{align*}
  &0_0 \neq \top_0\\ 
  &\forall\,x:n(x=0_0\vee x=\top_0)
 \end{align*}
 
 \item Using the abbreviation  $\ast_{n,m}(x,y)=x\ast_{n,m}y$, we write for $n,m\ge0$ sentences saying that ``$\ast_{n,m}$ is a biadditive function compatible with $h_n$'': 
 \begin{align*}
  &\forall\,x:n\forall\,y:n\forall\,z:m(((x+_ny)\ast_{nm}z)=(x\ast_{mn}z+_{n+m}y\ast_{nm}z)) \\
  &\forall\,x:n\forall\,y:m\forall\,z:m((x\ast_{mn}(y+_mz))=(x\ast_{nm}y+_{n+m}x\ast_{nm}z)) \\
  &\forall\,x:n\forall\,y:m (h_{n+m}(x\ast_{nm}y) = h_n(x)\ast_{nm} h_m(y))
 \end{align*}
 
 \item Sentences describing ``the induced ring with product induced by $\ast_{n,m}$, $n,m\ge0$'':
 \begin{align*}
  &\forall\, x:n\forall\, y:m\forall\, z:p((x\ast_{n,m}y)\ast_{(m+n),p}z=x\ast_{n,(m+p)}(y\ast_{m,p} z)) \\
  &\forall\, x:n\forall\, y:m(x\ast_{n,m}y=y\ast_{m,n}x)
 \end{align*}
 
 \item For $n\ge1$, sentences saying that ``$h_n=\top_1\ast_{1n}\_$'':
 \begin{align*}
  &\forall\,x:n(h_n(x)=\top_1\ast_{1n}x)
 \end{align*}
\end{enumerate}

Now we are in a position to  define another  version of Igr:

\begin{defn}[Inductive Graded Rings Second Version]\label{igr3}
An \textbf{inductive graded ring} (or \textbf{(Igr)} for short) is a model for $T(\mathcal L_{igr})$, or in other words, a $\mathcal L_{igr}$-structure $\mathcal R$ such that $\mathcal R\models_{\mathcal L_{igr}}T(\mathcal L_{igr})$. We denote by $\mbox{Igr}_2$ the category of $\mathcal L_{igr}$-structures and $\mathcal L_{igr}$-morphisms.
\end{defn}

Again, after some straightforward calculations we can check:
\begin{teo}
 The categories $\mbox{Igr}$, $\mbox{Igr}_2$ are equivalent.
\end{teo}

\begin{rem} \label{igr-re}
Following a well-known procedure, it is possible to correspond theories on poly-sorted first-order languages with theories on traditional (single-sorted)  first-order languages in such a way that the corresponding categories of models are equivalent. This allows a useful interchanging between model-theoretic results, in both directions. In particular, in the following, we will freely interchange the three notions of Igr indicated in this section.
\end{rem}




Theorem \ref{res0} gives a hint that the category of Igr is a good abstract environment for studying questions of "quadratic flavour". So a better understanding of categories of Igr's  and its applications to  quadratic forms theories is the main purpose of the next sections in this work.

\section{The First Properties of Igr}

In this section we discuss the theory of Igr's. Constructions like products, limits, colimits, ideals, quotients, kernel and image are not new and are obtained in a very straightforward manner (basically, putting those structures available for rings in a "coordinatewise" fashion), then in order to gain speed, we will present these facts leaving more detailed proofs to the reader.

Denote: $p\mathbb{F}_2-mod$ the category of pointed $\mathbb{F}_2$-modules, $\mbox{Ring}$ the category of commutative rings with unity and morphism that preserves these units and $\mbox{Ring}_2$ the full subcategory of the associative $\mathbb F_2$-algebras. We have a functorial correspondence   $\mbox{Ring}_2 \to \mbox{Igr}$, given by the following diagram:
$$\xymatrix@!=2pc{\ar@{} [dr] | {\mapsto}
A\ar[d]_{f} & \mathbb F_2\ar[r]^{!}\ar[d]_{id} & 
A\ar[r]^{id}\ar[d]_{f} & A\ar[r]^{id}\ar[d]_{f} & ...\ar[r]^{id} & A\ar[r]^{id}\ar[d]_{f} &  ... \\ 
B & \mathbb F_2\ar[r]^{!} & B\ar[r]^{id} & B\ar[r]^{id} & ...\ar[r]^{id} & B\ar[r]^{id} & ...}$$

Here $A$ is a $p\mathbb{F}_2-mod$ where $\top_n =1, n \geq 1$ and $\top_0 =1 \in \mathbb{F}_2$.

\begin{defn}\label{trivialigr}
The \textbf{trivial graded ring functor} $\mathbb T:\mbox{Ring}_2\rightarrow\mbox{Igr}$ is the functor defined for $f:A\rightarrow B$ by $T(A)_0:=\mathbb F_2$, $T(f)_0:=id_{\mathbb F_2}$ and for all $n\ge1$ we set $T(A)_n=A$ and $T(f)_n:=f$. 
\end{defn}

\begin{defn}\label{f2alg}
 We define the \textbf{associated $\mathbb F_2$-algebra functor} $\mathbb A:\mbox{Igr}\rightarrow\mbox{Ring}_2$ is the functor defined for $f:R\rightarrow S$ by 
 $$\mathbb A(R):=R_{\mathbb A}=\varinjlim\limits_{n\ge0}R_n\mbox{ and }\mathbb A(f)=f_{\mathbb A}:=\varinjlim\limits_{n\ge0}f_n.$$
\end{defn}

More explicitly, $\mathbb A(R)=(R_{\mathbb A},0,1,+_{\mathbb A},\cdot)$, where
\begin{enumerate}[i -]
 \item $R_{\mathbb A}=\varinjlim\limits_{n\ge0}R_n$,
 \item $0=[(0,0)]$ and $1=[(1,0)]$,
 \item given $[(a_n,n)],[(b_m,m)]\in R_{\mathbb A}$ and setting $d\ge m,n$ we have
 $$[(a_n,n)]+[(b_m,m)]=[(h_{nd}(a_n)+h_{md}(b_m),d)]$$
 \item given $[(a_n,n)],[(b_m,m)]\in R_{\mathbb A}$, we have
 $$[(a_n,n)]\cdot[(b_m,m)]=[(a_n\ast_{nm}b_m,n+m)].$$
\end{enumerate}

\begin{prop}\label{propadj1}
 $ $
 \begin{enumerate}[i -]
  \item The functor $\mathbb A$ is the left adjunct to $\mathbb T$.
 \item The functor $\mathbb T$ is full and faithful.
  \item The composite functor $\mathbb A\circ\mathbb T$ is naturally isomorphic to the functor $1_{\mbox{Ring}_2}$.
 \end{enumerate}
\end{prop}
\begin{proof}
Let $R\in Igr$. We have
$$\mathbb T(\mathbb A(R))=\mathbb T\left(\varinjlim\limits_{m\ge0}R_m\right).$$
In other words, for all $n\ge1$
$$\mathbb T\left(\varinjlim\limits_{m\ge0}R_m\right)_n:= \varinjlim\limits_{m\ge0}R_m.$$
Then, for all $n\ge1$ we have a canonical embedding
$$\eta(R)_n:R_n\rightarrow\varinjlim\limits_{m\ge0}R_m=\mathbb T\left(\varinjlim\limits_{m\ge0}R_m\right)_n,$$
providing a morphism
$$\eta(R):R\rightarrow\varinjlim\limits_{m\ge0}R_m=\mathbb T\left(\varinjlim\limits_{m\ge0}R_m\right).$$
For $f\in\mbox{Igr}(R,S)$, taking $n\ge1$ we have a commutative diagram
$$\xymatrix@!=5pc{R_n\ar[r]^{f_n}\ar[d]_{\eta(R)_n} & S_n\ar[d]^{\eta(S)_n} \\
**[l]\varinjlim\limits_{m\ge0}R_m\ar[r]_{\varinjlim\limits_{m\ge0}f_m} & **[r]\varinjlim\limits_{m\ge0}S_m}$$
with the convention that $\eta(R)_0=id_{\mathbb F_2}$. Then it is legitimate the definition of a natural transformation $\eta:1_{\mbox{Igr}}\rightarrow\mathbb T\circ\mathbb A$ given by the rule $R\mapsto\eta(R)$.

Now let $A\in Ring_2$ and $g\in Ring_2(R,\mathbb T(A))$. Then for each $n\ge0$, there is a morphism $g_n:R_n\rightarrow\mathbb T(A)_n=A$ and by the universal property of inductive limit we get a morphism
$$\varinjlim\limits_{m\ge0}g_n:\varinjlim\limits_{m\ge0}R_m\rightarrow A.$$
In fact, $\varinjlim\limits_{m\ge0}g_n=\mathbb A(g)$.

Now, using the fact that $\eta(R)_n$ is the morphism induced by the inductive limit we have for all $n\ge0$ the following commutative diagram
$$\xymatrix@!=5pc{
R_n\ar[dr]_{g_n}\ar[r]^{\eta(B)_n} & **[r]\varinjlim\limits_{m\ge0}R_m\ar[d]^{\varinjlim\limits_{m\ge0}g_n} \\
& A}$$
In other words, $\eta(B)_n$ is the canonical morphism commuting the diagram
$$\xymatrix@!=5pc{
R_n\ar[dr]_{g_n}\ar[r]^{\eta(B)_n} & **[r]\mathbb T(\mathbb A(R))\ar[d]^{\mathbb T(\mathbb A(g_n))} \\
& \mathbb T(A)}$$
and hence, $\mathbb A$ is the left adjoint of $\mathbb T$, proving item (i). By the very definition of $\mathbb A$ and $\mathbb T$ we get item (iii), and using Proposition \ref{3.2.2borceux} we get item (ii).
\end{proof}

Using Proposition \ref{3.2.2borceux} (and its dual version) we get the following Corollary.
\begin{cor}
 $ $
 \begin{enumerate}[i -]
  \item $\mathbb T:\mbox{Ring}_2\rightarrow\mbox{Igr}$ preserves all projective limits. 
  
  \item If $I$ is such that $\mbox{Igr}$ is $I$-inductively complete then for $\{A_i\}_{i\in I}$ in $\mbox{Igr}$ we have
$$\varinjlim\limits_{i\in I}A_i\cong{\mathbb A}\left(\varinjlim\limits_{i\in I}\mathbb T(A_i)\right).$$

  \item $\mathbb F_2\in\mbox{Ring}_2$ is the initial object in $\mbox{Ring}_2$.
  
  \item $0\in\mbox{Ring}_2$ is the terminal object in $\mbox{Ring}_2$.

  \item $\mathbb T(\mathbb F_2)$ is the initial object in $\mbox{Igr}$.
  
  \item $\mathbb T(0)$ is the terminal object in $\mbox{Igr}$.
 \end{enumerate}
\end{cor}

Now we discuss (essentially) the limits and colimits in $Igr$. Fix a non-empty set $I$ and let $\{(R_i,\top_i,h_i)\}_{i\in I}$ be a family of Igr's. We start with the construction of the Igr-product 
$$R=\prod_{i\in I}R_i.$$ 
For this, we define $R_0\cong\mathbb F_2$ and for all $n\ge1$, we define
$$R_n:=\prod_{i\in I}(R_i)_n\mbox{ and }\top_n:=\prod_{i\in I}(\top_i)_n.$$
In the sequel, we define $h_0:\mathbb F_2\rightarrow R_1$ as the only possible morphism and for $n\ge1$, we define $h_n:R_n\rightarrow R_{n+1}$ by
$$h_n:=\prod_{i\in I}(h_i)_n.$$

\begin{defn} \label{boohull-df}

\begin{enumerate}[i- ]
$ $
    \item 
    The {\bf space of orderings}, $X_R$, of the Igr $R$, is the set of Igr-morphisms $Igr(R, \mathbb{T}(\mathbb{F}_2)$. By the Proposition \ref{propadj1}.(i), we have a natural bijection $Igr(R, \mathbb{T}(\mathbb{F}_2) \cong Ring_2(\mathbb{A}(R), \mathbb{F}_2)$, thus considering the discrete topologies on the $\mathbb{F}_2$-algebras $\mathbb{A}(R), \mathbb{F}_2)$ and transporting the boolean topology in  $Ring_2(\mathbb{A}(R), \mathbb{F}_2)$, we obtain a boolean topology on the space of orderings $X_R = Igr(R, \mathbb{T}(\mathbb{F}_2))$.
    
\item The {\bf boolean hull}, $B(R)$, of the Igr $R$, is the boolean ring canonically associated to the space of orderings of $R$ by Stone duality: $B(R) := {\mathcal C}(X_R, \mathbb{F}_2)$.

\item A Igr $R$ is called {\bf formally real} if $X_R \neq \emptyset$ (or, equivalently, if $B(R) \neq 0$).
    \end{enumerate}
\end{defn}

\begin{prop}\label{fixigr1}
 Let $I$ be a non-empty set and $\{(R_i,h_i)\}_{i\in I}$ be a family of Igr's. Then
 $$R=\prod_{i\in I}R_i$$
 with the above rules is an Igr. Moreover it is the product in the category $\mbox{Igr}$.
\end{prop}
\begin{proof}
Using Definition \ref{igr1} is straightforward to verify that $(R,\top_n,h_n)$ is an Igr. Note that for each $i\in I$, we have an epimorphism $\pi_i:R\rightarrow R_i$ given by the following rules: for each $n\ge0$ and each $(x_i)_{i\in I}\in R_n$, we define
$$(\pi_i)_n((x_i)_{i\in I}):=x_i.$$
Now, let $(Q,\{q_i\}_{i\in I})$ be another pair with $Q$ being an Igr and $q_i:Q\rightarrow R_i$ being a morphism for each $i\in I$. Given $i\in I$ and $n\ge0$, since $R_n:=\prod_{i\in I}(R_i)_n$ is the product in the category of pointed $\mathbb F_2$-modules, we have an unique morphism $(q)_n:(Q)_n\rightarrow(R)_n$ such that $(\pi_i)_n\circ(q)_n=(q_i)_n$. Set $q_n:= ((q_i)_{i \in I})_n$. By construction, $q$ is the unique Igr-morphism such that $\pi_i\circ q=q_i$, completing the proof that $R$ is in fact the product in the category $\mbox{Igr}$.
\end{proof}

\begin{prop}
$ $
\begin{enumerate}[i- ]
\item    Let ${R}$ be an $Igr$ and let $X \subseteq R =
\underset{n\in\mathbb{N}}\bigoplus R_{n}$. Then  there exists the {\bf  inductive graded subring generated by $X$} (notation : $[X]\overset{i_{{X}}}\hookrightarrow R$): this is the least inductive graded subring of ${R}$ such that $\forall n \in \mathbb{N}$, $X\cap {R}_{n}\subseteq [ {X} ]_{n}$.
\item Let $\mathcal{I}$ be a small category and $\mathcal{R} : \mathcal{R} \to Igr$ be a diagram. Then there exists $\varprojlim_{i\in \mathcal{I}}{\mathcal{R}_i}$ in the category $Igr$.
\end{enumerate}
\end{prop}
\begin{proof}
$ $
\begin{enumerate}[i- ]
    \item  It is enough consider $S_X$ , the $\mathbb{F}_2$-subalgebra of $(\bigoplus_{n \in \mathbb{N}}R_n,\ast)$ generated by $X\cup\{ \top_{1} \} \subseteq \bigoplus_{n \in \mathbb{N}}R_n$ and set $\forall n \in \mathbb{N}$, $[{X}]_{n} := s_x\cap{R}_{n}$.
\item Just define $\varprojlim_{i\in \mathcal{I}}  \mathcal{R}_i$ as the inductive graded subring of $\prod_{i \in obj({\mathcal I})} \mathcal{R}_i$ generated by $X_D = \bigoplus_{n \in \mathbb{N} }  X_n$ and  $X_n := \varprojlim_{i \in \mathcal{I}} (\mathcal{R}_i)_n$ (projective limit of pointed $\mathbb{F}_2$-algebras).
\end{enumerate}
\end{proof}

Now we construct the Igr-tensor product of a \underline{finite family} of Igr's, $\{R_i : i \in I\}$
$$R=\bigotimes\limits_{i\in I}R_i.$$ 
For this, we define $R_0\cong\mathbb F_2$ and for all $n\ge1$, we define
$$R_n:=\bigotimes\limits_{i\in I}(R_i)_n,$$
$$(\otimes_{i \in I} a_i) \ast_{n,k} (\otimes_{i \in I} b_i) := \otimes_{i \in I} (a_i \ast^i_{n,k} b_i)$$
$$\mbox{ and }\top_n:=\otimes_{i\in I}(\top_i)_n.$$
In particular, if $I = \emptyset$, then $R_n = \{0\}, n \geq 1$. In the sequel, we define $h_0:\mathbb F_2\rightarrow R_1$ as the only possible morphism and for $n\ge1$, we define $h_n:R_n\rightarrow R_{n+1}$ by
$$h_n:=\bigotimes\limits_{i\in I}(h_i)_n.$$
In other words, for a generator $\bigotimes_{i\in I}x_i\in R_n$, we have
$$h_n\left(\otimes_{i\in I}x_i\right):=\bigotimes_{i\in I}(h_i)_n(x_i).$$

\begin{prop}\label{fixigr2}
 Let $I$ be a finite set and $\{(R_i,h_i)\}_{i\in I}$ be a family of Igr's. Then
 $$R=\bigotimes_{i\in I}R_i$$
 with the above rules is an Igr. Moreover it is the coproduct in the category $\mbox{Igr}$.
\end{prop}

Now suppose that $(I, \leq)$ is an upward directed poset and that $((R_i,h_i),\varphi_{ij})_{i\le j\in I}$ is an inductive system of Igr's. We define the inductive limit 
$$R=\varinjlim_{i\in I}R_i$$
by the following: for all $n\ge0$ define
$$R_n:=\varinjlim_{i\in I}(R_i)_n.$$
Note that
$$R_0:=\varinjlim_{i\in I}(R_i)_0\cong\varinjlim_{i\in I}\mathbb F_2\cong\mathbb F_2.$$
In the sequel, for $n\ge1$ we define $h_n:R_n\rightarrow R_{n+1}$ by
$$h_n:=\varinjlim_{i\in I}(h_i)_n.$$

\begin{prop}\label{fixigr4}
 Let $(I, \leq)$ is an upward directed poset and $((R_i,h_i),\varphi_{ij})_{i\in I}$ be a directed family of Igr's. Then
 $$R=\varinjlim_{i\in I}R_i$$
 with the above rules is an Igr. Moreover, it is the inductive limit in the category $\mbox{Igr}$.
\end{prop}

\begin{prop} The general coproduct (general tensor product) of a family $\{R_i :i \in I\}$ in the category $Igr$ is given by the combination of constructions: 
$$ \bigotimes_{i\in I}R_i := \varinjlim_{I' \in P_{fin}(I)} \bigotimes_{i\in I'}R_i. $$    
\end{prop}

After discussing  directed inductive colimits and coproducts, we will deal with ideals, quotients, and coequalizers.

\begin{defn}
Given $R\in\mbox{Igr}$ and $(J_n)_{n\ge0}$ where $J_n\subseteq R_n$ for all $n\ge0$. We say that $J$ is a \textbf{graded ideal} of $R$ where 
$$J:=\bigoplus_{n\ge0}J_n\subseteq\bigoplus_{n\ge0}R_n$$ 
is an ideal of $(R,\ast)$.
\end{defn}

In particular, for all $n\ge0$, $J_n\subseteq R_n$ is a graded $\mathbb F_2$-submodule of $(R_n,+_n,0_n)$. For each $X\subseteq R$, there exists the ideal generated by $X$, denoted by $\langle X\rangle$. It is the smaller graded ideal of $R$ such that for all $n\ge0$, $(X\cap R_n)\subseteq[X]_n$. For this, just consider $\langle X\rangle$, the ideal of $(R,\ast)$ generated by $X\subseteq R$ and define $\langle X\rangle_n:=\langle X\rangle\cap R_n$.

\begin{defn}
Let $R,S$ be Igr's and $f:R\rightarrow S$ be a morphism. We define the \textbf{kernel} of $f$, notation $\mbox{Ker}(f)$ by
$$\mbox{Ker}(f)_n:=\left\lbrace x\in R_n:f_n(x)=0\right\rbrace$$
and \textbf{image} of $f$, notation $\mbox{Im}(f)$ by
$$\mbox{Im}(f)_n:=\left\lbrace f_n(x)\in S_n:x\in R_n\right\rbrace.$$
\end{defn}

Of course, $\mbox{Ker}(f)\subseteq R$ is an ideal  and $\mbox{Im}(f)\subseteq S$ is an Igr.

Given $R\in\mbox{Igr}$ and $J=(J_n)_{n\ge0}$ a graded ideal of $R$, we define $R/J\in\mbox{Igr}$, the {\bf quotient inductive graded ring of $R$ by $J$}: for all $n\ge0$, $(R/J)_n:=R_n/J_n$, where the distinguished element is $\top_n+_nJ_n$. We have a canonical projection $q_J:R\rightarrow R/J$, ``coordinatewise surjective'' and therefore, an $\mbox{Igr}$-epimorphism.  

\begin{prop}[Homomorphism Theorem]
 Let $R,S$ be Igr's and $f:R\rightarrow S$ be a morphism. Then there exist an unique monomorphism $\overline f:R/\mbox{Ker}(f)\rightarrow S$ commuting the following diagram:
 $$\xymatrix@!=5pc{R\ar[r]^f\ar[d]_{q} & S \\ R/\mbox{Ker}(f)\ar[ur]_{\overline f} & }$$
 where $q$ is the canonical projection. In particular $R/\mbox{Ker}(f)\cong\mbox{Im}(f)$.
\end{prop}

\begin{prop} Let $R \underset{g}{\overset{f}{\rightrightarrows}} S$ be $Igr$-morphisms and consider $q_J : S \to S/J$ the quotient morphism where $J := \langle X \rangle$ is the graded ideal generated by $X_n := \{ f_n(a) - g_n(a): a \in R_n\}, n zin \mathbb{N}$. Then $q_J$ is the coequalizer of $f, g$. 
\end{prop}

\begin{prop} 
Given $R,S\in\mbox{Igr}$ and $f\in\mbox{Igr}(R,S)$.
\begin{enumerate}[i -]
    \item $f$ is a $\mbox{Igr}$-monomorphism whenever for all $n\ge0$ $f_n:R_n\rightarrow S_n$ is a monomorphism of pointed $\mathbb F_2$-modules iff for all $n\ge0$, $f_n:R_n\rightarrow S_n$ is an injective homomorphism of pointed $\mathbb F_2$-modules. 
    
    \item $f$ is a $\mbox{Igr}$-epimorphism whenever for all $n\ge0$ $f_n:R_n\rightarrow S_n$ is a epimorphism of pointed $\mathbb F_2$-modules iff for all $n\ge0$, $f_n:R_n\rightarrow S_n$ is a surjective homomorphism of pointed $\mathbb F_2$-modules. 
    
    \item $f$ is a $\mbox{Igr}$-isomorphism iff for all $n\ge0$ $f_n:R_n\rightarrow S_n$ is a isomorphism of pointed $\mathbb F_2$-modules iff for all $n\ge0$, $f_n:R_n\rightarrow S_n$ is a bijective homomorphism of pointed $\mathbb F_2$-modules. 
\end{enumerate}
\end{prop}

\begin{defn}
 We denote $\mbox{Igr}_{fin}$ the full subcategory of $\mbox{Igr}$ such that
 $$\mbox{Obj}(\mbox{Igr}_{fin})=\left\lbrace R\in\mbox{Obj}(\mbox{Igr}):\left| R_n\right|<\omega\mbox{ for all }n\ge1\right\rbrace.$$
\end{defn}

\begin{rem}
    Of course,
$$\left\lbrace R\in\mbox{Obj}(\mbox{Igr}):\left|\bigoplus\limits_{n\ge1}R_n\right|<\omega\right\rbrace\ne\mbox{Obj}(\mbox{Igr}_{fin}),$$
for example, in \ref{ex1}(a), if $F$ is a Euclidian field (for instance, any real closed field), then $\underset{{n \in \mathbb{N}}}{\bigoplus}I^nF/I^{n+1}F$ $  \cong \mathbb F_2[x]$, thus the graded Witt ring of $F$ (see definition \ref{graded Witt ring functor}) $W_*(F)\in\mbox{Obj}(\mbox{Igr}_{fin})$ but $\mathbb F_2[x]$ is not finite.
\end{rem}

\section{Relevant subcategories of $\mbox{Igr}$}

The aim of this section is to define  subcategories of Igr that mimetize the following two central aspects of K-theories:
\begin{enumerate}
    \item The K-theory graded ring is "generated" by $K_1$;
    \item The K-theory graded ring is defined by some convenient quotient of a graded tensor algebra.
\end{enumerate}
Our desired category will be the intersection of two subcategories. The first one is obtained after we define the \textbf{graded subring generated by the level 1} functor
 $$\mathbbm1:\mbox{Igr}\rightarrow\mbox{Igr}.$$ 
 We define it as follow: for an object $R=((R_n)_{n\ge0},(h_n)_{n\ge0},\ast_{nm})$,
\begin{enumerate}[i -]
 \item $\mathbbm1(R)_0:=R_0\cong\mathbb F_2$,
 \item $\mathbbm1(R)_1:=R_1$,
 \item for $n\ge2$,
 \begin{align*}
  \mathbbm1(R)_n&:=\{x\in R_n:x=\sum\limits^{r}_{j=1}a_{1j}\ast_{11}...\ast_{11}a_{nj}, \\
  &\mbox{with }a_{ij}\in R_1,\,1\le i\le n,\,1\le j\le r\mbox{ for some }r\ge1\}.
 \end{align*}
 Note that for all $n\ge2$, $R_n$ is generated by the expressions of type 
 $$d_1\ast_{11}d_2\ast_{11}...\ast_{11}d_n,\, d_i\in R_1,\, i=1,...,n.$$
\end{enumerate}

 Of course, $\mathbbm1(R)$ provides an inclusion $\iota_{\mathbbm1(R)}:\mathbbm1(R)\rightarrow R$ in the obvious way.

On the morphisms, for $f\in\mbox{Igr}(R,S)$, we define $\mathbbm1(f)\in\mbox{Igr}(\mathbbm1(R),\mathbbm1(S))$ by the restriction $\mathbbm1(f)=f\upharpoonleft_{\mathbbm1(R)}$. In other words, $\mathbbm1(f)$ is the only Igr-morphisms that makes the following diagram commute:
$$\xymatrix@!=6pc{\mathbbm1(R)\ar[r]^{\iota_{\mathbbm1(R)}}\ar[d]_{\mathbbm1(f)} & R\ar[d]^{f} \\ 
\mathbbm1(S)\ar[r]_{\iota_{\mathbbm1(S)}} & S}$$

\begin{defn}\label{level1}
We denote $\mbox{Igr}_{\mathbbm1}$ the full subcategory of $\mbox{Igr}$ such that 
$$\mbox{Obj}(\mbox{Igr}_{\mathbbm1})=\{R\in\mbox{Igr}:\iota_{\mathbbm1(R)}:\mathbbm1(R)\rightarrow R\mbox{ is an isomorphism}\}.$$
\end{defn}

\begin{ex} 
$ $
\begin{enumerate}[i -]
    \item If $A$ is a $\mathbb{F}_2$-algebra, then $\mathbb{T}(A) \in obj(\mbox{Igr}_{\mathbbm1})$.

\item If $F$ is an hyperbolic hyperfield, then $k_*(F)  \in obj(\mbox{Igr}_{\mathbbm1})$.

\item If $F$ is a special hyperfield (equivalently, $G = F \setminus \{0\}$ is a special group), then the graduate Witt ring of $F$ (definition \ref{graded Witt ring functor}) $W_*(F)  \in obj(\mbox{Igr}_{\mathbbm1})$.

\item If $F$ is a field with $char(F) \neq 2$, then, by a known result of Vladimir Voevodski, 
$$\mathcal{H}^*(Gal(F^s|F), \{\pm 1\})  \in obj(\mbox{Igr}_{\mathbbm1}).$$
    \end{enumerate}
\end{ex}

\begin{prop}
 $ $
 \begin{enumerate}[i -]
  \item For each $R\in\mbox{Igr}$ we have that $\iota_{\mathbbm1(\mathbbm1(R))}:\mathbbm1(\mathbbm1(R))\rightarrow\mathbbm1(R)$ is the identity arrow.
  
  \item $\mathbbm1\circ\mathbbm1=\mathbbm1$.
  
  \item The functor $\mathbbm1:\mbox{Igr}\rightarrow\mbox{Igr}_{\mathbbm1}$ is the right adjoint of the inclusion functor $j_{\mathbbm1}:\mbox{Igr}_{\mathbbm1}\rightarrow\mbox{Igr}$.
  \item  $j_{\mathbbm1}:\mbox{Igr}_{\mathbbm1}\rightarrow\mbox{Igr}$ creates inductive limits and to obtain the projective limits in $\mbox{Igr}_{\mathbbm1}$ is sufficient restrict the projective limits obtained in $\mbox{Igr}$:
$$\varprojlim\limits_{i\in I}R_i\cong\left(\varprojlim\limits_{i\in I}j_{\mathbbm1}(R_i)\right)_{\mathbbm1}\xrightarrow{\varprojlim\limits_{i\in I}j_{\mathbbm1}(R_i)}\varprojlim\limits_{i\in I}j_{\mathbbm1}(R_i).$$
 \end{enumerate}
\end{prop}
\begin{proof}
Similar to Proposition \ref{propadj1}.
\end{proof}

 Now we define the second subcategory. We define the \textbf{quotient graded ring functor} 
 $$\mathcal Q:\mbox{Igr}\rightarrow\mbox{Igr}$$ 
 as follow: for a object $R=((R_n)_{n\ge0},(h_n)_{n\ge0},\ast_{nm})$, $\mathcal Q(R):=R/T$, where $T=(T_n)_{n\ge0}$ is the ideal generated by $\{(\top_1+_1a)\ast_{11}a\in R_2:a\in R_1\}$. More explicit,
\begin{enumerate}[i -]
 \item $T_0:=\{0_0\}\subseteq R_0$,
 \item $T_1:=\{0_1\}\subseteq R_1$,
 \item for $n\ge2$, $T_n\subseteq R_n$ is the pointed $\mathbb F_2$-submodule generated by
 \begin{align*}
  \{x\in R_n:x&=y_l\ast_{l1}(\top_1+_1a_1)\ast_{11}a_1\ast_{1r}z_r, \\
  &\mbox{with }a_1\in R_1,\,y_l\in R_l,\,z_r\in R_r,\,l+r=n-2\}.
 \end{align*}
 Of course, $\mathcal Q(R)$ provides a projection $\pi_R:R\rightarrow\mathcal Q(R)$ in the obvious way.
\end{enumerate}

On the morphisms, for $f\in\mbox{Igr}(R,S)$, we define $\mathcal Q(f)\in\mbox{Igr}(\mathcal Q(R),\mathcal Q(S))$ by the only Igr-morphisms that makes the following diagram commute:
$$\xymatrix@!=6pc{R\ar[r]^{\pi_R}\ar[d]_{f} & \mathcal Q(R)\ar[d]^{\mathcal Q(f)} \\ 
S\ar[r]_{\pi_{S}} & \mathcal Q(S)}$$

\begin{defn}\label{quotop}
We denote $\mbox{Igr}_h$ the full subcategory of $\mbox{Igr}$ such that 
$$\mbox{Obj}(\mbox{Igr}_h)=\{R\in\mbox{Igr}:\pi_R:R\rightarrow\mathcal Q(R)\mbox{ is an isomorphism}\}.$$
\end{defn}

\begin{rem}
    Note that $R \in obj(Igr_h)$ iff for each $a \in R_1$, $a \ast_{11} \top_1 = a \ast_{11} a \in R_2$. Each $R$ satisfying this condition is, in some sense, ``hyperbolic'' (see Proposition \ref{prespechf}): this is the motivation of the index ``h''.
\end{rem}

\begin{ex} 
\begin{enumerate}[i- ]
    \item Let $A$ be  a $\mathbb{F}_2$-algebra. Then $\mathbb{T}(A) \in obj(\mbox{Igr}_{h})$ iff $A$ is a boolean ring (i.e., $\forall a \in A, a^{2} = a$).

\item If $F$ is an hyperbolic hyperfield, then $k_*(F)  \in obj(\mbox{Igr}_{h})$.

\item If $F$ is a special hyperfield (equivalently, $G = F \setminus \{0\}$ is a special group), then $W_*(F)  \in obj(\mbox{Igr}_{h})$.

\item If $F$ is a field with $char(F) \neq 2$, then $\mathcal{H}^*(Gal(F^s|F), \{\pm 1\})  \in obj(\mbox{Igr}_{h})$.
\end{enumerate}
\end{ex}

\begin{prop}
 $ $
 \begin{enumerate}[i -]
  \item For each $R\in\mbox{Igr}$ we have that $\pi_{\mathcal Q(R)}:\mathcal Q(R)\rightarrow\mathcal Q(\mathcal Q(R))$ is an isomorphism.
  
  \item $\mathcal Q\circ\mathcal Q=\mathcal Q$.
  
  \item The functor $\mathcal Q:\mbox{Igr}\rightarrow\mbox{Igr}_h$ is the left adjoint of the inclusion functor $j_q:\mbox{Igr}_{\mathcal Q}\rightarrow\mbox{Igr}$.
  \item $j_q:\mbox{Igr}_h\rightarrow\mbox{Igr}$ creates projective limits and to obtain the inductive limits in $\mbox{Igr}_h$ is sufficient restrict the inductive limits obtained in $\mbox{Igr}$:
$$\varinjlim\limits_{i\in I}j_q(R_i)\xrightarrow{\varinjlim\limits_{i\in I}j_q(R_i)}\left(\varinjlim\limits_{i\in I}j_q(R_i)\right)_{\mathcal Q}\cong\varinjlim\limits_{i\in I}R_i.$$
  Moreover, $j_q:\mbox{Igr}_h\rightarrow\mbox{Igr}$ creates filtered inductive limits and quotients by graded ideals.
 \end{enumerate}
\end{prop}

Are examples of inductive graded rings in $Igr_{+}$: (i) $\mathbb{T}(A)$, where $A$ is a boolean ring; (ii) $k_*(F)$, where $F$ is an hyperbolic hyperfield; (iii) $W_*(F)$, where $F$ is an special hyperfield; (iv)  $\mathcal{H}^*(Gal(F^s|F), \{\pm 1\}) $, where  $F$ is a field with $char(F) \neq 2$.

\begin{defn}[The Category $\mbox{Igr}_+$]\label{igr+}
 We denote by $\mbox{Igr}_+$ the full subcategory of $\mbox{Igr}$ such that 
 $$\mbox{Obj}(\mbox{Igr}_+)=\mbox{Obj}(\mbox{Igr}_{\mathbbm1})\cap\mbox{Obj}(\mbox{Igr}_h).$$ 
 We denote by $j_+:\mbox{Igr}_+\rightarrow\mbox{Igr}$ the inclusion functor.
\end{defn}

\begin{rem} 
\begin{enumerate} [i- ]
    \item  Note that the notion of an Igr, $R$,  be in the subcategory $Igr_h$ can be axiomatized by a first-order (finitary) sentence in   $L$,  the polysorted language for Igr's described in the previous Chapter: ($\forall a:1 , a\ast_{11}a = \top_1 \ast_{11} a$). On the other hand, the concepts $R \in \mbox{Igr}_{\mathbbm1}$ and $R \in $ $Igr_+$ are axiomatized by  $L_{\omega_1,\omega}$-sentences.
    \item Note that the subcategory $Igr_+ \hookrightarrow Igr$ is closed by filtered inductive limits.
\end{enumerate}
    
\end{rem}

In order to think of an object in $\mbox{Igr}_+$ as a graded ring of "K-theoretic type", we make the following convention.

\begin{defn}[Exponential and Logarithm of an Igr]\label{igrlog}
Let $R\in\mbox{Igr}_+$ and write $R_1$ \textbf{multiplicatively} by $(\Gamma(R),\cdot,1,-1)$, i.e, fix an isomorphism $e_R:R_1\rightarrow\Gamma(R)$ in order that $e_R(\top)=-1$ and $e_R(a+b)=a\cdot b$. Such isomorphism $e_R$ is called \textbf{exponential} of $R$ and $l_R=e_R^{-1}$ is called \textbf{logarithm} of $R$. In this sense, we can write $R_1=\{l(a):a\in \Gamma(R)\}$. We also denote $l(a)\ast_{11}l(b)$ simply by $l(a)l(b)$, $a,b\in \Gamma(R)$. We drop the superscript and write just $e,l$ when the context allows it.
\end{defn}

Using Definitions \ref{igr+}, \ref{igrlog} (and of course, Definitions \ref{level1} and \ref{quotop} with an argument similar to the used in Lemma \ref{bp1}) we have the following properties.

\begin{lem}[First Properties]\label{igr_+first}
Let $R\in\mbox{Igr}_+$.
\begin{enumerate}[i -]
    \item $l(1)=0$.
    \item For all $n\ge1$, $\eta\in R_n$ is generated by $l(a_1)...l(a_n)$ with $a_1,...,a_n\in\Gamma(R)$.
    \item $l(a)l(-a)=0$ and $l(a)l(a)=l(-1)l(a)$ for all $a\in\Gamma(R)$.
    \item $l(a)l(b)=l(b)l(a)$ for all $a,b\in\Gamma(R)$.
    \item For every $a_1,...,a_n\in\Gamma(R)$ and every permutation $\sigma\in S_n$,
 $$l(a_1)...l(a_i)...l(a_n)=\mbox{sgn}(\sigma)l(a_{\sigma 1})...l(a_{\sigma n})\mbox{ in }R_n.$$
    \item For all $\xi\in R_n$, $\eta\in R_n$,
    $$\xi\eta=\eta\xi.$$
    \item For all $n\ge1$, $$h_n(l(a_1)...l(a_n))=l(-1)l(a_1)...l(a_n).$$
\end{enumerate}
\end{lem}

\begin{prop} \label{igr_+prop} Let $R\in\mbox{Igr}_+$
\begin{enumerate}[i- ]
    \item For each $n \in \mathbb{N}$ and each $x\in R_n$, $x\ast_{n,n} x = \top_n \ast_{n,n} x \in R_{2n}$.
    \item $\mathbb{A}(R) = \varinjlim_{n \in \mathbb{N}} R_n$ is a boolean ring (or, equivalently, $\mathbb{T}(\mathbb{A}(R)) \in \mbox{Igr}_+)$.
\end{enumerate}
\end{prop}
\begin{proof}
$ $
\begin{enumerate}[i- ]
\item  The property is clear if $n =0$. If $n \geq 1$, then the property can be verified by induction on the number of generators $k \geq 1, x = \sum_{i =1}^k  a_{1,i} \ast_{1,1} a_{2,i} \ast_{1,1} \cdots \ast_{1,1}a_{n,i} \in R_n$: if $k =1$, then   note that 
\begin{align*}
 x\ast_{n, n} x &= (a_{1} \ast a_{2} \ast \cdots \ast a_{n}) \ast (a_{1} \ast a_{2} \ast \cdots \ast a_{n}) \\
 &= (a_{1} \ast a_{1}) \ast (a_2 \ast a_2) \ast \cdots (a_n \ast a_{n})=(\top_1 \ast a_1) \ast (\top_1 \ast a_2) \ast \cdots \ast (\top_1 \ast a_n)\\
 &= (\top_n) \ast (a_{1} \ast a_{2} \ast \cdots \ast a_{n});
\end{align*}
if $k > 1$, write $x = y + z$, where $y, z \in R_n$ are have $<k$ generator and then, by induction,
\begin{align*}
    x \ast_{n, n} x &= (y+z) \ast_{n,n}(y+z) = y\ast_{n,n} y   + y\ast_{n,n} z + z\ast_{n,n} y   +  z\ast_{n,n} z  \\
    &= y\ast_{n,n} y + z\ast_{n,n} z= \top_n \ast_{n,n} y + \top_n \ast_{n,n} z \\
    &= \top_n \ast_{n,n}(y+z) = \top_n \ast_{n,n} x
\end{align*}
\item This follows directly from item (i) and the definition of the ring structure in $\mathbb{A}(R) = \varinjlim_{n \in \mathbb{N}} R_n$. 
    \end{enumerate}
\end{proof}

By the previous Proposition and the universal property of the boolean hull of an Igr (Definition \ref{boohull-df}.(ii)), we obtain:

\begin{cor}\label{igr+co} Let $R\in\mbox{Igr}_+$. Then:
\begin{enumerate}[i- ]
\item $X_{\mathbb{T}(\mathbb{A}(R))} \approx X_R$.
\item $\mathbb{A}(R) \cong B(R)$.
\end{enumerate}
\end{cor}

\begin{lem}
 $ $
 \begin{enumerate}[i -]
  \item Given $R\in\mbox{Igr}_{\mathbbm1}$, $S\in\mbox{Igr}$ and $f:S\rightarrow j_{\mathbbm1}(R)$, we have:
  $f$ is coordinatewise surjective iff $f_1:S_1\rightarrow R_1$ is a surjective morphism of pointed $\mathbb{F}_2$-modules.
  
  \item Given $R\in\mbox{Igr}_{\mathbbm1}$, $S\in\mbox{Igr}$ and $f,h\in\mbox{Igr}(j_{\mathbbm1}(R),S)$, we have $f=h$ if and only if
  $f_1=h_1$.
 \end{enumerate}
\end{lem}

Let $R,S\in\mbox{Igr}$. The inclusion function $\iota_R:\mathbbm 1(R)\rightarrow R$ and projection function $\pi_R:R\rightarrow\mathcal Q(R)$ induces respective natural transformations $\iota:\mathbbm1\Rightarrow1_{Igr}$ and $\pi:1_{Igr}\Rightarrow\mathcal Q$. Moreover, we have a natural transformation $can:\mathcal Q\mathbbm1\Rightarrow\mathbbm1\mathcal Q$ given by the rule $can_n(l(a_1)...l(a_n)):=l(a_1)...l(a_n)$, $n\ge1$. ($can_n$ is well defined and is an isomorphism basically because both $\mathcal Q\mathbbm1(R)$ and $\mathbbm1\mathcal Q(R)$ are generated in level 1 by $R_1$ and both graded rings satisfies the relation $l(a)l(-a)=0$).

We have another immediate consequence of the previous results (and adjunctions):
\begin{lem}
$ $
 \begin{enumerate}[i -]
  \item For all $R\in\mbox{Igr}_h$, $\mathbbm1(R)\in\mbox{Igr}_+$ and $\mbox{can}_R$ is an isomorphism.
  
  \item For all $R\in\mbox{Igr}_{\mathbbm1}$, $\mathcal Q(R)\in\mbox{Igr}_+$ and $\mbox{can}_R$ is an isomorphism.
  
  \item To get projective limits in $\mbox{Igr}_+$ is enough to restrict the projective limits obtained in $\mbox{Igr}$:
  $$\varprojlim\limits_{i\in I}R_i\cong\mathbbm1\left(\varprojlim\limits_{i\in I}j_+(R_i)\right).$$
  
  \item To get inductive limits in $\mbox{Igr}_+$ is enough to restrict the inductive limits obtained in $\mbox{Igr}$:
  $$\varinjlim\limits_{i\in I}R_i\cong\mathcal Q\left(\varinjlim\limits_{i\in I}j_+(R_i)\right).$$
 \end{enumerate}
\end{lem}

\section{Examples and Constructions of Quadratic Interest}

\begin{defn}
 A \textbf{filtered ring} is a tuple $A=(A,(J_n)_{n\ge0},+,\cdot,0,1)$ where:
 \begin{enumerate}[i -]
  \item $(A,+,\cdot,0,1)$ is a commutative ring with unit.
  \item $J_0=A$ and for all $n\ge1$, $J_n\subseteq A$ is an ideal.
  \item For all $n,m\ge0$, $n\le m\Rightarrow J_n\supseteq J_m$.
  \item For all $n,m\ge0$, $J_n\cdot J_m\subseteq J_{n+m}$.
  \item $J_0/J_1\cong\mathbb F_2$ (then $2=1+1\in J_1$).
  \item For all $n\ge0$, $J_n/J_{n+1}$ is a group of exponent 2 (then $2\cdot J_n\subseteq J_{n+1}$ and $2^n\in J_n$).
 \end{enumerate}
 A \textbf{morphism} $f:A\rightarrow A'$ of filtered rings is a ring homomorphism such that $f(J_n)\subseteq J'_n$. The category of filtered rings will be denoted by $\mbox{FRing}$.
\end{defn}

\begin{defn}\label{gradfilt}
 We define the \textbf{inductive graded ring associated functor} 
 $$Grad:\mbox{FRing}\rightarrow\mbox{Igr}$$ 
 for $f:\mbox{FRing}(A,B)$ as follow: $Grad(A):=((Grad(A)_n)_{n\ge0},(t_n)_{n\ge0},\ast)\in\mbox{Igr}$ is the igr where
\begin{enumerate}[i -]
 \item For all $n\ge0$, $Grad(A)_n:=(J_n/J_{n+1},+_n,0_n,\top_n)$ is the exponent 2 group with distinguished element $\top_n:=2^n+J_{n+1}$.

 \item For all $n\ge0$, $t_n:Grad(A)_n\rightarrow Grad(A)_{n+1}$ is defined by $t_n:=2\cdot\_$, i.e, 
$$\mbox{For all }a+J_{n+1}\in J_n/J_{n+1},\,t_n(a+J_{n+1}):=2\cdot a+J_{n+2}\in J_{n+1}/J_{n+2}.$$
Observe that $t_n(\top_n)=\top_{n+1}$, i.e, $t_n$ is a 
morphism of pointed $\mathbb F_2$-modules.

 \item For all $n,m\ge0$ the biadditive function $\ast_{nm}:Grad(A)_n\times Grad(A)_m\rightarrow Grad(A)_{n+m}$ is defined by the rule
 $$(a_n+J_{n+1})\ast_{mn}(b_m+J_{m+1})=a_n\cdot b_m+J_{n+m+1}\in J_{n+m}/J_{n+m+1}.$$
 The group $A_g:=\bigoplus_{n\ge0}Grad(A)_n$ of exponent 2 and the induced application $\ast:A_g\times A_g\rightarrow A_g$ are such that $(A_g,\ast)$ is a commutative ring with unit $\top_1=(2+J_2)\in J_1/J_2$.

 \item For all $n\ge1$, $t_n=\top_1\ast_{1n}\_$.
\end{enumerate}

The morphism $Grad(f)\in\mbox{Igr}(Grad(A),Grad(A'))$ is defined by the following rules: for all $n\ge0$, $f_n:Grad(A)_n\rightarrow Grad(A')_n$ is given by
$$f_n(a+J_{n+1}):=f_n(a)+J'_{n+1}.$$
Note that $f_n$ a homomorphism of $\mathbb F_2$-pointed modules and $\bigoplus_{n\ge0}f_n:(A_g,\ast)\rightarrow(A'_g,\ast)$ is a homomorphism of graded rings with unit.
\end{defn}

\begin{defn}\label{igrcont}
 The \textbf{functor of graded ring of continuous functions} over a space $X$
 $$\mathcal C(X,\_):\mbox{Igr}\rightarrow\mbox{Igr}$$ 
 is the functor defined for $f:R\rightarrow S$ by
\begin{enumerate}[i -]
 \item $\mathcal C(X,R)_0:=R_0\cong\mathbb F_2$,
 \item for all $n\ge1$, $\mathcal C(X,R)_n:=\mathcal C(X,R_n)$ as a pointed $\mathbb{F}_2$-module,
 \item for all $n,m\ge0$, $\ast^X_{nm}:\mathcal C(X,R_n)\times\mathcal C(X,R_m)\rightarrow\mathcal C(X,R_{n+m})$ is given by $(\alpha_n,\beta_m)\mapsto\alpha_n\ast^X_{nm}\beta_m$, where for $x\in X$,
$$\alpha_n\ast^X_{nm}\beta_m(x)=\alpha_n(x)\ast_{nm}\beta_m(x)\in R_{n+m}.$$
  \item $\mathcal C(X,f)_0:=f_0$ as an homomorphism of pointed $\mathbb{F}_2$-modules $R_0 \to S_0$.
 \item for all $n\ge1$, $\mathcal C(X,f)_n:=\mathcal C(X,f_n):=f_n\circ\_$  $\in p\mathbb{F}_2-mod( \mathcal C(X,R_n), \mathcal C(X,S_n))$.
\end{enumerate}
\end{defn}

\begin{rem} Let $X$ be a topological space and  let $R\in\mbox{Igr}_{\mathbbm1}$. Note that if $X$ is compact or $R \in Igr_{fin}$, then $ \mathcal C(X,R)\in\mbox{Igr}_{\mathbbm1}$.
\end{rem}

\begin{defn}\label{sgfilt}
 We define the \textbf{continuous function filtered ring functor} 
 $$\mathcal C:SG\rightarrow\mbox{FRing}$$ 
 as follow: first, consider the functor $\mathcal C(X_{\_},\mathbb Z):SG\rightarrow\mbox{Ring}$, composition of the (contravariant) functors ``associated ordering space'' $X_{\_}:SG\rightarrow\mbox{Top}^{op}$ and ``continuous functions in $\mathbb Z$ ring'' $\mathcal C(\_,\mathbb Z):\mbox{Top}^{op}\rightarrow\mbox{Ring}$ (here $\mathbb Z$ is endowed with the discrete topology).

Now we define the functor $\mathcal C:SG\rightarrow\mbox{FRing}$: given a special group $G\in SG$, we define 
$$\mathcal C(G):=(R(G),(J_n(G))_{n\ge0},+,\cdot,0,1)$$ 
where 
\begin{enumerate}[i -]
 \item $(R(G),+,\cdot,0,1)$ is the subring of $\mathcal C(X_G,\mathbb Z)$ of continuous functions of constant parity, i.e, $R(G):=J_0(G)\xrightarrow{i_0(G)}\mathcal C(X_G,\mathbb Z)$ is the image of the monomorphism of rings with unit
$$j_0(G):\mathcal C(X_G,2\mathbb Z)\cup\mathcal C(X_G,2\mathbb Z+1)\rightarrow\mathcal C(X_G,\mathbb Z).$$

 \item For all $n\ge1$, $J_n(G)\xrightarrow{i_n(G)}J_0(G)$ is the ideal of $R(G)$ (and also of $\mathcal C(X_G,\mathbb Z)$) that is the image of the monomorphism of abelian groups
$$j_n(G):\mathcal C(X_G,2^n\mathbb Z)\rightarrow \mathcal C(X_G,2\mathbb Z)\cup\mathcal C(X_G,2\mathbb Z+1).$$
\end{enumerate}

We also have $J_0(G)/J_1(G)\cong\mathbb F_2$ and for all $n,m\ge0$:
\begin{enumerate}[a -]
 \item If $n\ge m$ then $J_n(G)\supseteq J_m(G)$;
 \item $J_n(G)\cdot J_m(G)\subseteq J_{n+m}(G)$;
 \item $2J_n(G)=J_{n+1}(G)\Rightarrow J_n(G)/J_{n+1}(G)$ is an exponent 2 group.
\end{enumerate}

On the morphisms, for $f\in SG(G,G')$, we define $\mathcal C(f)\in\mbox{FRing}(\mathcal C(G),\mathcal C(G'))$ by
$$\mathcal C(f)(h)=\mathcal C(X_f,\mathbb Z)(h)$$
for $h\in\mathcal C(G)$. $\mathcal C(f)$ is well-defined because $\mathcal C(f)\in\mbox{Ring}(\mathcal C(G),\mathcal C(G'))$ and for all $n\ge0$, 
$$\mathcal C(f)(J_n(G))\subseteq J_n(G').$$
\end{defn}

\begin{defn}
We define the \textbf{continuous function graded ring functor} by 
 $$Grad\circ\mathcal C:SG\rightarrow\mbox{Igr}.$$ 
\end{defn}

 For convenience, we describe this functor now: given $G\in\mbox{SG}$,
$$Grad(\mathcal C(G)):=((Grad(\mathcal C(G))_n)_{n\ge0},(t_n)_{n\ge0},\cdot)$$
where:
\begin{enumerate}[i -]
 \item $Grad(\mathcal C(G))_n:=(J_n(G)/J_{n+1}(G),\cdot,0\cdot J_{n+1}(G),2^nJ_{n+1}(G))$, where $2\in\mathcal C(X_G,\mathbb Z)$ is the constant function of value $2\in2\mathbb Z\subseteq\mathbb Z$.
 
 \item For all $n\ge0$, $J_n(G)/J_{n+1}(G)\xrightarrow{t_2=2\cdot\_}J_{n+1}(G)/J_{n+2}(G)$.
 
 \item For all $n,m\ge0$, $\ast_{nm}:J_n(G)/J_{n+1}(G)\times J_m(G)/J_{m+1}(G)\rightarrow J_{n+m}(G)/J_{n+m+1}(G)$ is given by
 $$(h_n+J_{n+1}(G))\ast_{nm}(k_m+J_{m+1}(G))=h_nk_m+J_{n+m+1}(G).$$
\end{enumerate}

On the morphisms, given $f\in SG(G,G')$, we have that 
$$Grad(\mathcal C(f))=(Grad(\mathcal C(f))_n)_{n\ge0}\in\mbox{Igr}(Grad(\mathcal C(G), Grad(\mathcal C(G')),$$ 
where for all $n\ge0$, $Grad(\mathcal C(f))_n:Grad(\mathcal C(G))_n\rightarrow Grad(\mathcal C(G'))_n$ is such that
$$Grad(\mathcal C(f))_n(h+J_{n+1}(G))=\mathcal C(f)(h)+J'_{n+1}(G').$$

\begin{prop}
 $ $
 \begin{enumerate}[a -]
  \item There is a natural isomorphism $\theta:Grad\circ\mathcal C\xrightarrow{\cong}\mathbb T\circ\mathcal C(X_{\_},\mathbb F_2)$. In particular, for all $G\in SG$, $Grad(\mathcal C(G))\in\mbox{Igr}_+$.
  
  \item For all $0< n\le m<\omega$, $2^{m-n}\cdot\_:J_n(G)/J_{n+1}(G)\rightarrow J_m/J_{m+1}(G)$ is an isomorphism of groups of exponent 2.
  
  \item For all $n\ge1$, there is an isomorphism of groups of exponent 2 
$$\theta_n(G):J_n(G)/J_{n+1}(G)\xrightarrow{\cong}\mathcal 
C(X_G,\mathbb F_2),$$
given by the rule
  $$\theta_n(h+J_n(G))(\sigma):=h_n(\sigma)/2^n\in\mathcal C(X_G,\mathbb Z/2\mathbb Z).$$
  
  \item For all $0< n\le m<\omega$ the following diagram commute:
  $$\xymatrix@!=4pc{J_n(G)/J_{n+1}(G)\ar[rr]^{2^{m-n}\cdot\_}\ar[dr]_{\theta_n(G)} & & J_m(G)/J_{m+1}(G)\ar[dl]^{\theta_m(G)} \\ & 
\mathcal C(X_G,\mathbb F_2) & }$$
 \end{enumerate}
\end{prop}

\begin{defn}\label{filtered Witt ring functor} 
 We define the \textbf{filtered Witt ring functor} 
 $$\mathcal W:SG\rightarrow\mbox{FRing}$$ 
 for $f\in\mbox{SG}(G,H)$ as follow: given a special group $G\in SG$, we define 
$$\mathcal W(G):=(W(G),I^n(G)_{n\ge0},\oplus,\otimes,\langle\rangle,\langle1\rangle)$$ 
where for all $n\ge0$, $I^n(G)$ is the $n$-th power of the fundamental ideal 
$$I(G):=\{\varphi\in W(G):\dim_2(\varphi)=0\}.$$
 
 We define $\mathcal W(f)\in\mbox{FRing}(\mathcal W(G),\mathcal W(H))$ by the rule $\mathcal W(f)(\varphi):=f\star\varphi$.
\end{defn}

$\mathcal W(G)$ is a filtered commutative ring with unit because:

\begin{enumerate}[i -]
 \item $(W(G),\oplus,\otimes,\langle\rangle,\langle1\rangle)\in\mbox{Ring}$.
 \item For all $n\ge0$, $I^n(G)\subseteq W(G)$ is an ideal.
 \item For all $n,m\ge0$, $n\le m\Rightarrow I^n(G)\supseteq I^m(G)$.
 \item For all $n,m\ge0$, $I^n(G)\otimes I(G)\subseteq I^{n+m}(G)$.
 \item $I^0(G):=W(G)$.
 \item $I^0(G)/I^1(G)\cong\mathbb F_2$.
 \item For all $n\ge0$, $(I^n(G)/I^{n+1}(G),\oplus,\langle\rangle)$ is a group of exponent 2 with distinguished element $2^n + I^{n+1}(G)$, where $2^n = \otimes_{i<n} \langle 1, 1 \rangle$.
 \end{enumerate}

\begin{defn} \label{graded Witt ring functor}
 We define the \textbf{graded Witt ring} functor
 $$\mbox{Grad}\circ\mathcal W:SG\rightarrow\mbox{Igr}.$$
%
%
%
\end{defn}

We register, again, the following result:
\begin{prop}
 For each $G\in SG$ we have $\mbox{Grad}(\mathcal W(G))\in\mbox{Igr}_+$.
\end{prop}

For each commutative ring with unit $A$, we have 
$$t(A)=\{a\in A:\mbox{ exists }n\ge0\mbox{ with }n\cdot a=0\}\subseteq A$$
is an ideal (the torsion ideal of $A$). The association $A\mapsto A/t(A)$ is the component on the objects of an endofunctor of $\mbox{Ring}$.

For each $G\in SG$ we have a ring homomorphism with unit $\mbox{sgn}_G:W(G)\rightarrow\mathcal C(X_G,\mathbb Z)$ given by the rule
$$\mbox{sgn}_G(\langle a_0,...,a_{n-1}\rangle)(\sigma):=\sum^{n-1}_{i=0}\sigma(a_i).$$
The Pfister's Local-Global principle says that $\mbox{sgn}_G$ induces a monomorphism 
$$\mbox{rsgn}_G:W(G)/t(W(G))\rightarrow\mathcal C(X_G,\mathbb Z).$$

For each $G\in SG$ we have $\mbox{sgn}_G(W(G))\subseteq\mathcal C(X_G,2\mathbb Z)\cup\mathcal C(X_G,2\mathbb Z+1)$ (since the signatures of classes of forms has the same parity of its dimension) and for all $n\ge1$, $\mbox{sgn}_G(I^n(G))\subseteq\mathcal C(X_G,2^n\mathbb Z)$ (since $I^n(G)$ is the abelian subgroup of $W(G)$ generated by classes of Pfister forms of dimension $2^n$).

$\mbox{sgn}:\mathcal W\rightarrow\mathcal C$ (respectively $\mbox{rsgn}:\mathcal W/t(\mathcal W)\rightarrow\mathcal C$) is the natural transformation between functors
$$\xymatrix@!=3pc{SG\ar@<+.5ex>[r]^{\mathcal W}\ar@<-.5ex>[r]_{\mathcal C} & \mbox{FRing}}$$
that provide natural transformations between functors $\xymatrix{SG\ar@<+.5ex>[r]\ar@<-.5ex>[r] & \mbox{Igr}}$:
\begin{align*}
 \mbox{Grad}\cdot\mbox{sgn}&:\mbox{Grad}\circ\mathcal W\rightarrow\mbox{Grad}\circ\mathcal C\mbox{, respectively} \\
 \mbox{Grad}\cdot\mbox{rsgn}&:\mbox{Grad}\circ(\mathcal W/t(\mathcal W))\rightarrow\mbox{Grad}\circ\mathcal C.
\end{align*}
Remember that [MC] ([LC]) and [WMC] ([WLC]) are conjectures about these natural transformations.

$\mathcal C$ is a particular case of $\mathcal W$ in the following sense: $\mathcal C:SG\rightarrow\mbox{FRing}$ is naturally isomorphic to the composition of functors $SG\xrightarrow{\gamma\circ\beta}SG\xrightarrow{\mathcal W}\mbox{FRing}$.

\section{The adjunction between $\mbox{PSG}$ and $\mbox{Igr}_h$}

By the very definition of the K-theory of hyperfields (with the notations in Theorem \ref{3.3ktmultiadap}) we define the following functor.

 \begin{defn}[K-theories Functors]
  With the notations of Theorem \ref{3.3ktmultiadap} we have a functors $k:\mathcal{HMF}\rightarrow\mbox{Igr}_+$, $k:\mathcal{PSMF}\rightarrow\mbox{Igr}_+$ induced by the reduced K-theory for hyperfields.
 \end{defn}
 
 Now, let $R\in\mbox{Igr}_h$. We define a hyperfield $(\Gamma(R),+,-.\cdot,0,1)$ by the following: firstly, fix an exponential isomorphism $e_R:(R_1,+_1,0_1,\top_1)\rightarrow (G(R),\cdot,1,-1)$ (in agreement with Definition \ref{igrlog}). This isomorphism makes, for example, an element $a\ast_{11}(\top_1+b)\in R_2$, $a,b\in R_1$ take the form $(l_R(x))\ast_{11}(l_R((-1)\cdot y))\in R_2$, $x,y\in G(R)$. By an abuse of notation, we simply write $l_R(x)l_R(-y)\in R_2$, $x,y\in G(R)$. In this sense, an element in $Q_2$ has the form $l_R(x)l_R(-x)$, $x\in \Gamma(R)$, and we can extend this terminology for all $Q_n$, $n\ge2$ (see Definition \ref{quotop}, and Lemma \ref{igr_+first}).
 
 Now, let $\Gamma(R):=G(R)\cup\{0\}$ and for $a,b\in\Gamma(R)$ we define
 \begin{align}\label{gammahyper}
    -a&:=(-1)\cdot a,\nonumber\\
    a\cdot0&=0\cdot a:=0,\nonumber\\
    a+0&=0+a=\{a\},\nonumber \\
    a+(-a)&=\Gamma(R),\nonumber \\
    \mbox{for }&a,b\ne0,a\ne-b\mbox{ define }\nonumber\\
    a+b&:=\{c\in\Gamma(R):\mbox{there exist }d\in G(R)\mbox{ such that } \nonumber\\
    &a\cdot b=c\cdot d\in G(R)\mbox{ and }l_R(a)l_R(b)=l_R(c)l_R(d)\in R_2\}.
\end{align}

\begin{prop}\label{prespechf}
With the above rules, $(\Gamma(R),+,-.\cdot,0,1)$ is a pre-special hyperfield.
\end{prop}
\begin{proof}
We will verify the conditions of Definition \ref{defn:multiring}. Note that by the definition of multivalued sum once we proof that $\Gamma(R)$ is an hyperfield, it will be hyperbolic. In order to prove that $(\Gamma(R),+,-.\cdot,0,1)$ is a multigroup we follow the steps below. Here we use freely the properties in Lemma \ref{igr_+first}.
 \begin{enumerate}[i -]
   \item Commutativity and $(a\in b+0)\Leftrightarrow(a=b)$ are direct consequence of the definition of multivaluated sum and the fact that $l_R(a)l_R(b)=l_R(b)l_R(a)$. 
   
  \item We will prove that if $c\in a+b$, then $a\in c-b$ and $b\in c-a$.
  
  If $a=0$ (or $b=0$) or $a=-b$, then $c\in a+b$ means $c=a$ or $c\in a-a$. In both cases we get $a\in c-b$ and $b\in c-a$.
  
  Now suppose $a,b\ne0$ with $a\ne-b$. Let $c\in a+b$. Then $a\cdot b=c\cdot d$ and $l_R(a)l_R(b)=l_R(c)l_R(d)\in R_2$ for some $d\in G(R)$. Since $G(R)$ is a multiplicative group of exponent 2, we have $a\cdot d=b\cdot c$ (and hence $a\cdot(-d)=c\cdot(-b)$). Note that 
  \begin{align*}
      l_R(a)l_R(-d)&=l_R(a)l_R(-abc)=l_R(a)l_R(bc)=l_R(a)l_R(b)+l_R(a)l_R(c)\\
      &=l_R(c)l_R(d)+l_R(a)l_R(c)=l_R(c)l_R(d)+l_R(c)l_R(a)=l_R(c)l_R(ad).
  \end{align*}
  Similarly,
  \begin{align*}
      l_R(b)l_R(-c)&=l_R(b)l_R(-abd)=l_R(b)l_R(ad)=l_R(b)l_R(a)+l_R(b)l_R(d) \\
      &=l_R(a)l_R(b)+l_R(b)l_R(d)=l_R(c)l_R(d)+l_R(b)l_R(d) \\ &=l_R(bc)l_R(d)=l_R(ad)l_R(d).
  \end{align*}
  Then
  \begin{align*}
      l_R(a)l_R(-d)-l_R(b)l_R(-c)&=l_R(c)l_R(ad)-l_R(ad)l_R(d)= \\
      &=l_R(c)l_R(ad)-l_R(d)l_R(ad)=l_R(-cd)l_R(ad).
  \end{align*}
  But
  \begin{align*}
      l_R(-cd)l_R(ad)&=l_R(-cd)l_R(a)+l_R(-cd)l_R(d)= \\
      &=l_R(-cd)l_R(a)+l_R(c)l_R(d)=l_R(a)l_R(-cd)+l_R(a)l_R(b) \\
      &=l_R(a)l_R(-bcd)=l_R(a)l_R(-a)=0.
  \end{align*}
  Then 
  $$l_R(a)l_R(-d)=l_R(b)l_R(-c),$$
  proving that $a\in b-c$. Similarly we prove that $b\in -c+a$.
  
  \item Since $(G(R),\cdot,1)$ is an abelian group, we conclude that $(\Gamma(R),\cdot,1)$ is a commutative monoid. Beyond this, every nonzero element $a\in\Gamma(R)$ is such that $a^2=1$.
  
  \item $a\cdot0=0$ for all $a\in\Gamma(R)$ is direct from definition.
  
  \item For the distributive property, let $a,b,d\in\Gamma(R)$ and consider $x\in d(a+b)$. We need to prove that
  \begin{align*}
      \tag{*}x\in d\cdot a+d\cdot b.
  \end{align*}
  It is the case if $0\in\{a,b,d\}$ or if $b=-a$. Now suppose $a,b,d\ne0$ with $b\ne-a$. Then there exist $y\in G(R)$ such that $x=dy$ and $y\in a+b$. Moreover, there exist some $z\in G(R)$ such that $y\cdot z=a\cdot b$ and $l_R(y)l_R(z)=l_R(a)l_R(b)$. 
  
  If $0\in\{a,b,d\}$ or if $b=-a$ there is nothing to prove. Now suppose $a,b,d\ne0$ with $b\ne-a$. Therefore $(dy)\cdot(dz)=(da)\cdot(db)$ and 
    \begin{align*}
      l_R(dy)l_R(dz)&=l_R(d)l_R(d)+l_R(d)l_R(z)+l_R(d)l_R(y)+l_R(y)l_R(z) \\
      &=l_R(d)l_R(d)+l_R(d)[l_R(z)+l_R(y)]+l_R(y)l_R(z) \\
      &=l_R(d)l_R(d)+l_R(d)l_R(yz)+l_R(y)l_R(z) \\
      &=l_R(d)l_R(d)+l_R(d)l_R(ab)+l_R(a)l_R(b) \\
      &=l_R(d)l_R(d)+l_R(d)l_R(a)+l_R(d)l_R(b)+l_R(a)l_R(b) \\
      &=l_R(da)l_R(db),
      \end{align*}
  so $l_R(dy)l_R(dz)=l_R(da)l_R(db)$. Hence we have $x=dy\in d\cdot a+d\cdot b$.
  
  \item Using distributivity we have that for all $a,b,c,d\in\Gamma(R)$
  $$d[(a+b)+c]=(da+db)+dc\mbox{ and }d[a+(b+c)]=da+(db+dc).$$
  In fact, if $x\in (a+b)+c$, then $x\in y+c$ for $y\in a+b$. Hence
  $$dx\in dy+dc\subseteq d(a+b)+dc=(da+db)+dc.$$
  Conversely, if $z\in(da+db)+dc$, then $z=w+dc$, for some $w\in da+db=d(a+b)$. But in this case, $w=dt$ for some $t\in a+b$. Then
  $$z\in dt+dc=d[t+c]\subseteq d[(a+b)+c].$$
  Similarly we prove that $d[a+(b+c)]=da+(db+dc)$.
  
  \item Let $a\in\Gamma(R)$ and $x,y\in1-a$. If $a=0$ or $a=1$ then we automatically have $x\cdot y\in 1-a$, so let $a\ne0$ and $a\ne1$. Then $x,y\in G(R)$ and there exist $p,q\in\Gamma(R)$ such that
  \begin{align*}
      x\cdot p=1\cdot a&\mbox{ and }l_R(x)l_R(p)=l_R(1)l_R(a)=0 \\
      y\cdot q=1\cdot a&\mbox{ and }l_R(y)l_R(q)=l_R(1)l_R(a)=0.
  \end{align*}
  Then $(xy)\cdot(pqa)=1\cdot a$ and
  \begin{align*}
      l_R(xy)l_R(pqa)
      &=l_R(xy)l_R(p)+l_R(xy)l_R(q)+l_R(xy)l_R(a) \\
      &=l_R(y)l_R(p)+l_R(x)l_R(q)+l_R(x)l_R(a)+l_R(y)l_R(a) \\
      &=l_R(y)l_R(pa)+l_R(x)l_R(qa) \\
      &=l_R(y)l_R(x)+l_R(x)l_R(y)=0.
  \end{align*}
  Then $xy\in1-a$, proving that $(1-a)(1-a)\subseteq(1-a)$. In particular, since $1\in1-a$, we have $(1-a)(1-a)=(1-a)$.

  \item Finally, to prove associativity, we use Theorem \ref{psgpsmfhell}. Let $\langle a,b\rangle\equiv\langle c,d\rangle$ the relation defined for $a,b,c,d\in\Gamma(R)\setminus\{0\}$ by
  $$\langle a,b\rangle\equiv\langle c,d\rangle\mbox{ iff }ab=cd\mbox{ and }l_R(a)l_R(b)=l_R(c)l_R(d).$$
  For $0\notin\{a,b,c,d\}$, $a\ne-b$ and $ab=cd$, we have
  $$a+b=c+d\mbox{ iff }\langle a,b\rangle\equiv\langle c,d\rangle.$$
  Using items (i)-(vii) we get that $(\Gamma(R)\setminus\{0\},\equiv,1,-1)$ is a pre-special group. Then by Theorem \ref{psgpsmfhell} we have that $M(\Gamma(R)\setminus\{0\})\cong\Gamma(R)$ is a pre-special hyperfield, and in particular, $(\Gamma(R)$ is associative.
 \end{enumerate}
\end{proof}

 \begin{defn}
  With the notations of Proposition \ref{prespechf} we have a functor $\Gamma:\mbox{Igr}_+\rightarrow\mbox{PSMF}$ defined by the following rules: for $R\in\mbox{Igr}_+$, $\Gamma(R)$ is the special hyperfield obtained in Proposition \ref{prespechf} and for $f\in\mbox{Igr}_+(R,S)$, $\Gamma(f):\Gamma(R)\rightarrow\Gamma(S)$ is the unique morphism such that the following diagram commute
  $$\xymatrix@!=4.5pc{R\ar[d]_{f_1}\ar[r]^{e_R} & \Gamma(R)\ar@{.>}[d]^{\Gamma(f)} \\ S\ar[r]_{e_S} & \Gamma(S)}$$
  In other words, for $x\in R$ we have
  $$\Gamma(f)(x)=(e_S\circ f_1\circ l_R)(x)=e_S(f_1(l_R(x))).$$
 \end{defn}
 
 \begin{teo}\label{psgadj}
 The functor $k:\mathcal{PSMF}\rightarrow\mbox{Igr}_+$ is the left adjoint of $\Gamma:\mbox{Igr}_+\rightarrow\mathcal{PSMF}$. The unity of the adjoint is the natural transformation $\phi:1_{\mathcal{PSMF}}\rightarrow\Gamma\circ k$ defined for $F\in\mathcal{PSMF}$ by $\phi_F=e_{k(F)}\circ\rho_F$.
 \end{teo}
 \begin{proof}
 We show that for all $f\in\mathcal{PSMF}(F,\Gamma(R))$ there is an unique $f^\sharp:\mbox{Igr}_+(k(F),R)$ such that $\Gamma(f^\sharp)\circ\phi_F=f$. Note that $\phi_F=e_{k(F)}\circ\rho_F$ is a group isomorphism (because $e_{k(F)}$ and $\rho_F$ are group isomorphisms).
 
 Let $f^\sharp_0:1_{\mathbb F_2}:\mathbb F_2\rightarrow\mathbb F_2$ and 
 $f^\sharp_1:=l_R\circ f\circ(\phi_F)^{-1}\circ e_{k(F)}:k_1(F)\rightarrow R_1$. For $n\ge2$, define $h_n:\prod^n_{i=1}k_1(F)\rightarrow R_n$ by the rule
 $$h_n(\rho(a_1),...,\rho(a_n)):=l_R(f(a_1))\ast...\ast l_R(f(a_n)).$$
 We have that $h_n$ is multilinear and by the Universal Property of tensor products we have an induced morphism $\bigotimes^n_{i=1}k_n(F)\rightarrow R_n$ defined on the generators by
 $$h_n(\rho(a_1)\otimes...\otimes\rho(a_n)):=l_R(f(a_1))\ast...\ast l_R(f(a_n)).$$
 Now let $\eta\in Q_n(F)$. Suppose without loss of generalities that $\eta=\rho(a_1)\otimes...\otimes\rho(a_n)$ with $a_1\in1-a_2$. Then $f(a_1)\in1-f(a_2)$ which imply $l_R(f(a_1))\in1-l_R(f(a_2))$. Since $R_n\in\mbox{Igr}_+$,
 \begin{align*}
     h_n(\eta):=h_n(\rho(a_1)\otimes...\otimes\rho(a_n))
     =l_R(f(a_1))\ast...\ast l_R(f(a_n))=0\in R_n.
 \end{align*}
 Then $h_n$ factors through $Q_n$, and we have an induced morphism $\overline h_n:k_n(F)\rightarrow R_n$. We set $f^{\sharp}_n:=\overline h_n$. In other words, $f^\sharp_n$ is defined on the generators by
 $$f^\sharp_n(\rho(a_1)...\rho(a_n)):=l_R(f(a_1))\ast...\ast l_R(f(a_n).$$
 Finally, we have
 \begin{align*}
   \Gamma(f^\sharp)\circ\phi_F
   &=[e_{R}\circ(f^\sharp_1)\circ e^{-1}_{k(F)}]\circ[e_{k(F)}\circ\rho_F]
   =e_{R}\circ(f^\sharp_1)\circ\rho_F \\
   &=e_{R}\circ[l_R\circ f\circ(\phi_F)^{-1}\circ e_{k(F)}]\circ\rho_F \\
   &=f\circ(\phi_F)^{-1}\circ[e_{k(F)}\circ\rho_F] \\
   &=f\circ(\phi_F)^{-1}\circ\phi_F=f.
 \end{align*}
 
 For the unicity, let $u,v\in\mbox{Igr}_+(k(F),R)$ such that $\Gamma(u)\circ\phi_F=\Gamma(v)\circ\phi_F$. Since $\phi_F$ is an isomorphism we have $u_1=v_1$ and since $k(F)\in\mbox{Igr}_+$ we have $u=v$.
 \end{proof}
 
 As we have already seen in Theorem \ref{psgadj}, there natural transformation $\phi_F:F\rightarrow\Gamma(k(F))$ is a group isomorphism. Now let $a,c,d\in F$ with $a\in c+d$. Then $\phi_F(a)\in\phi_F(c)+\phi_F(d)$, i.e, $\phi_F$ is a morphism of hyperfields. In fact, if $0\in\{a,c,d\}$ there is nothing to prove. Let $0\notin\{a,c,d\}$. To prove that $\phi_F(a)\in\phi_F(c)+\phi_F(d)$ we need to show that $\rho_F(a)\rho_F(acd)=\rho_F(c)\rho_F(d)$. In fact, from $a\in c+d$ we get $ac\in1+ad$, and then $\rho_F(ac)\rho_F(ad)=0$. Moreover
 \begin{align*}
     \rho_F(a)\rho_F(acd)+\rho_F(c)\rho_F(d)
     &=\rho_F(a)\rho_F(acd)+\rho_F(c)\rho_F(d)+\rho_F(ac)\rho_F(ad) \\
     &=\rho_F(a)\rho_F(ac)+\rho_F(a)\rho_F(d)+\rho_F(c)\rho_F(d)+\rho_F(ac)\rho_F(ad) \\
     &=[\rho_F(a)\rho_F(ac)+\rho_F(ac)\rho_F(ad)]+[\rho_F(a)\rho_F(d)+\rho_F(c)\rho_F(d)] \\
     &=\rho_F(d)\rho_F(ac)+\rho_F(d)\rho_F(ac)=0,
 \end{align*}
proving that $\phi_F(a)\in\phi_F(c)+\phi_F(d)$. Unfortunately we do not now if or where $\phi_F$ is a strong morphism. Then we propose the following definition.

\begin{defn}[The $k$ stability]
Let $F$ be a pre-special hyperfield. We say that $F$ is \textbf{$k$-stable} if $\phi_F:F\rightarrow\Gamma(F(G))$ is a strong morphism. Alternatively, $F$ is $k$-stable if for all $a,b,c,d\in\dot F$, if $ab=cd$ then $$\rho_F(a)\rho_f(b)=\rho_F(c)\rho_F(d)\mbox{ imply }ac\in1+cd.$$
\end{defn}

\begin{prop} Every PSG $G$ has a $k$-stable hull $G_{(k)}$ that satisfies the corresponding universal property . This is just given by 
$$G_{(k)} =\varinjlim_{n \in \mathbb{N}} (\Gamma \circ k)^n (G).$$
Thus the inclusion functor $PSG_{(k)} \hookrightarrow PSG$ has a left adjoint $(k) : PSG \to PSG_{(k)}$.
\end{prop}

We emphasize that if $G$ is $AP(3)$ special group, then $G$ is $k$-stable. In particular, every reduced special group is $k$-stable, and if $F$ is a field of characteristic not 2, then $G(F)$ is also $k$-stable. 

 In the next Chapter, it is established the 
 Arason-Pfister Hauptsatz (Theorem \ref{haup}) for {\bf every special group} $G$, (i.e., $G$ satisfies $AP(n)$ for each $n \in \mathbb{N}$.)

\begin{prop}
 $ $
 \begin{enumerate}[i -]
  \item For each $G\in SG$, $\Gamma(s_G):\Gamma(\mathcal K(G))\rightarrow\Gamma(\mbox{Grad}(\mathcal W(G)))$ is a PSG-isomorphism.
  
  \item For each $G\in\mathcal{RSG}$, $\kappa_G:G\rightarrow\Gamma(\mathcal K(G))$ is a PSG-isomorphism.
  
  \item For each $G\in\mathcal{RSG}$, $\omega_G:G\rightarrow\Gamma(\mbox{Grad}(\mathcal W(G)))$ is a PSG-isomorphism.
 \end{enumerate}
\end{prop}

\begin{prop}
 Let $G$ be a PSG. Are equivalent:
 \begin{enumerate}[i -]
  \item $G\in\mathcal{PSG}_{fin}$.
  \item $\mathcal K(G)\in\mbox{Igr}_{fin}$.
 \end{enumerate}
\end{prop}

\begin{prop}
 Let $G$ be a SG. Are equivalent:
 \begin{enumerate}[i -]
  \item $G\in SG_{fin}$.
  \item $\mathcal K(G)\in\mbox{Igr}_{fin}$.
  \item $(\mbox{Grad}\circ\mathcal W)(G)\in\mbox{Igr}_{fin}$.
 \end{enumerate}
\end{prop}

\begin{prop}
 The canonical arrow
 $$can:\varinjlim_{i\in I}\mathcal K(G_i)\rightarrow\mathcal K\left(\varinjlim_{i\in I}G_i\right)$$
 is an $\mbox{Igr}_+$-isomorphism as long as the $I$-colimits above exists.
\end{prop}

\begin{prop}
 The canonical arrow
 $$can:\mathcal K\left(\varprojlim_{i\in I}G_i\right)\rightarrow\varprojlim_{i\in I}\mathcal K(G_i)$$
 is an $\mbox{Igr}_+$-morphism pointwise surjective, as long as the $I$-colimits above exists.
\end{prop}

\begin{rem}
    In \cite{dickmann1998quadratic} there is an interesting analysis identifying the boolean hull of a special group $G$ (or special hyperfield $F = G \cup \{0\}$) with the boolean hull of the inductive graded rings $k_*(F), W_*(F) \in Igr_+$ (see the above Corollary \ref{igr+co}). It could be interesting to compare the space of orderings of $R \in Igr_{h}$ and of $\Gamma(R) \in \mathcal{PSMF}$.
\end{rem}

\section{Igr and Marshall's Conjecture}

Using the Boolean hull functor, M. Dickmann and F. Miraglia provide an encoding of Marshall's signature conjecture ([MC]) for reduced special groups  by the condition 
$$\langle 1, 1 \rangle \otimes - : I^n(G)/I^{n+1}(G) \to I^{n+1}(G)/I^{n+2}(G)$$
to be injective, for each $n \in \mathbb{N}$. In fact they introduce the  notion of a [SMC] reduced special group:
$$l(-1) \otimes - : k_{n}(G) \to k_{n+1}(G)$$ 
is injective, for each $n \in \mathbb{N}$. They  establish that,  [SMC] imply [MC], for every reduced special group $G$. Moreover (see 5.1 and 5.4 in \cite{dickmann2006algebraic}):
\begin{itemize}

    \item The inductive limit of [SMC] groups is [SMC].

    \item  The finite product of [SMC] groups is [SMC].
    
    \item $G(F)$ is [SMC], for every Pythagorean field $F$ (with ($char(F) \neq 2$).
\end{itemize}



\begin{prop}
$ $
 \begin{enumerate}[i -]
  \item $s:k\rightarrow\mbox{Grad}\circ\mathcal W$ is a ``surjective'' natural transformation, where for each  $G\in SG$ and all $n\ge1$, $s_n(G):K_n(G)\rightarrow I^n(G)/I^{n+1}(G)$ is given by the rule
  \begin{align*}
   s_n(G)\left(\sum^{s-1}_{i=0}l(g_{1,i})\otimes...\otimes l(g_{n,i})+\mathcal Q_n(G)\right):=
   \overline\bigotimes^{s-1}_{i=0}[\langle1,-g_{1,i}\rangle]\overline\otimes...\overline\otimes[\langle1,-g_{n,i}\rangle]\overline\otimes I^{n+1}(G).
  \end{align*}

  \item $r:\mbox{Grad}\circ\mathcal W\rightarrow k$ is a natural transformation, where for each  $G\in SG$ and all $n\ge1$,
   $r^n_G:I^n(G)/I^{n+2}(G)\rightarrow k_{2n-1}(G)$ is given by the rule
  \begin{align*}
   r_n(G)\left(\overline\bigotimes^{s-1}_{i=0}[\langle1,-g_{1,i}\rangle]\overline\otimes...\overline\otimes[\langle1,-g_{n,i}\rangle]\overline\otimes I^{n+1}(G)\right):= \\
   \sum^{s-1}_{i=0}l(-1)^{2^{n-1}-n}l(g_{1,i})\otimes...\otimes l(g_{n,i})+\mathcal Q_{2n-1}(G)
  \end{align*}
  
  \item For all $n\ge1$, $r_n(G)\circ s_n(G)=l(-1)^{2^{n-1}-n}\overline\otimes\_$.
  
  \item We have an isomorphism of pointed $\mathbb F_2$-modules: $s^1_G:k_1(G)\xrightarrow{\cong}I^1(G)/I^2(G)$,
  $s^2_G:k_2(G)\xrightarrow{\cong}I^2(G)/I^3(G)$.
  
  \item If $G$ is [SMC] Then $s_G:k(G)\rightarrow\mbox{Grad}\circ\mathcal W(G)$ is an isomorphism.
 \end{enumerate}
\end{prop}

$ $

We finish this chapter considering a general setting for ``Marshall's conjectures'', that includes the previous case of the Igr's $W_*(F), k_*(F)$ for special hyperfields $F$. 

 Let $R \in Igr_+$. The ideal, $nil(R)$, in the ring $\underset{n\in\mathbb{N}}\bigoplus R_{n}$, formed by all of its nilpotent elements, determines $N(R)$ a $Igr$-ideal of ${R}$, where
   $(N({R}))_{n} := nil(R)\cap{R}_{n}$, $\forall n \in \mathbb{N}$. Note that, by Proposition \ref{igr_+prop}, $(nil({R}))_{n} =\{ a \in {R}_{n} : \exists k \in \mathbb{N}\setminus\{0\} ( \top_{kn}\ast_{kn,n}a = 0_{(k+1)n} ) \}$ , $\forall n \in \mathbb{N}$.

\begin{rem}
Let $\rho : \mathbb{N} \to \mathbb{N} $ be an increasing function and define $(N_{\rho}(R))_n = \{ a \in {R}_{n} : \exists k \in \mathbb{N} ( \top_{\rho(n)}\ast_{\rho(n),n}a = 0_{\rho(n)+n} ) \}$ , $\forall n \in \mathbb{N}$. Then  $(N_{\rho}(R))_n$ is a subgroup of $R_n$ and, since $\rho(n+k) \geq \rho(n)$, we have $ (N_{\rho}(R))_n \ast_{n,k} R_k \subseteq (N_{\rho}(R))_{n+k}$. Summing up, $(N_{\rho}(R))_n)_{n \in \mathbb{N}}$ is an Igr-ideal.
\end{rem}

The following result is straightforward consequence of the Definitions and \ref{propadj1}, \ref{igr+co}.

\begin{prop}
 For each $R\in\mbox{Igr}_+$ are equivalent:
 \begin{enumerate}[i -]
  \item For all $n \leq m\in\mathbb{N}$, $\mbox{ker}(h_{nm})=\{0_n\}\in R_n$.
  \item The canonical morphism $R \to \mathbb{T}(\mathbb{A}(R))$ is pointwise injective.
  \item There exists a boolean ring $B$ and a pointwise injective Igr-morphism $R \to \mathbb{T}(B)$.
  
  Moreover, if $R \in Igr_{fin}$, these are equivalent to
  \item $N(R)\cong\mathbb{T}(0)\in \mbox{Igr}$.
 \end{enumerate}
Motivated by item (i), we use the abbreviation $\mbox{MC}(R)$ to say that  $R$ satisfies one (and hence all) of the above conditions.
\end{prop}



In the following, we fix a category of $L$-structures $\mathcal{A}$ that is closed under directed inductive limits and a functor  $F_* : \mathcal{A} \to  Igr_+$ be a functor that preserves directed inductive limits. Examples of such kind of functors are $k_* : \mathcal{HMF} \to  Igr_+$ and $W_* : \mathcal{HMF} \to  Igr_+$,  since such hyperfields can be conveniently described in the first-order relational language for multirings and it is closed under directed inductive limits. Related examples are the functors $k_* : SG \to  Igr_+$ and $W_* : SG \to  Igr_+$; note that $SG$ is a full subcategory of $L_{SG}-Str$ that is closed under directed inductive limits {\bf and} under arbitrary products.

\begin{prop}
    If $(I,\leq)$ is an upward directed poset and $\Gamma : (I, \leq) \to \mathcal{A}$ is such that: $MC(F_*(\Gamma(i)))$, for all $i \in I$, then $MC(F_*(\varinjlim_{i \in I}\Gamma(i)))$.
\end{prop}
\begin{proof}
    The hypothesis on $F_*$ and  the fact that the directed inductive limits in $Igr_+$ are pointwise, give us immediately that the mappings $h_n : F_n(\varinjlim_{i \in I}\Gamma(i)) \to F_{n+1}(\varinjlim_{i \in I}\Gamma(i)) $ are isomorphic to the injective maps $\varinjlim_{i \in I} h_n^i : \varinjlim_{i \in I}F_n(\Gamma(i)) \to \varinjlim_{i \in I} F_{n+1}(\Gamma(i)) $, for each $n \in \mathbb{N}$. Therefore it holds $$MC(F_*(\varinjlim_{i \in I}\Gamma(i)))$$.  
\end{proof}

\begin{cor} Let $F \subseteq P(I)$ be a filter and let $\{M_i : i \in I\}$ be a family of (non-empty) $L$-structures in $\mathcal{A}$. Suppose that $\mathcal{A}$ is closed under products and suppose that holds $MC(F_*(\prod_{i \in J} M_i))$, for each $J \in F$. Then holds $MC(F_*(\prod_{i \in J} M_i/F))$.
\end{cor} 
\begin{proof}
   This follows from the preceding result since, by a well-known model-theoretic result due to D. Ellerman (\cite{ellerman1974sheaves}), any reduced product of a family of (non-empty) $L$-structures, $\{M_i : i \in I\}$, module a filter $F \subseteq P(I)$, is canonically isomorphic to an upward directed inductive limit, $\varinjlim_{J \in F} (\prod_{i \in J}M_i) \cong   (\prod_{i \in I}M_i)/F$.
\end{proof}

\begin{prop} Let  $F_* : \mathcal{A} \to  Igr_+$ preserves pure embeddings. More precisely, if $M, M' \in \mathcal{A}$ and  $j : M \to M'$ is a pure $L$-embedding, then $F_*(j) : F_*(M) \to F_*(M')$ is a pure morphism of Igr's (described in the first-order polysorted language for Igr's).
\end{prop}

\begin{proof}
  This follows from the well known characterization result:
  
  {\bf Fact:} Let $L'$ be a first-order language and $f :   A \to B$ be an $L'$-homomorphism. Then are equivalent
  \begin{itemize}
      \item  $f :   A \to B$ is a pure $L'$-embedding.
      \item There exists an elementary $L'$-embedding $e : A \to C$ and a $L'$-homomorphism $h : B \to C$, such that $e = h \circ f$.
      \item There exists an ultrapower $A^I/U$ and  a $L'$-homomorphism $g : B \to A^I/U$, such that $\delta_A^{(I,U)} = g \circ f$, where $\delta_A ^{(I,U)} : A \to A^I/U$ is the diagonal (elementary) $L'$-embedding. 
  \end{itemize}

  Since the morphism $j : M \to M'$ is a pure embedding, by the Fact there exists an ultrapower ${M}^I/U$ and  a $L$-homomorphism $g : M' \to M^I/U$, such that $\delta^M_{(I,U)} = g \circ j$, where $\delta_M^{(I,U)} : M \to M^I/U$ is the diagonal (elementary) $L$-embedding.  

Since we have a canonical isomorphism $can : \varinjlim_{J \in U} M^J \overset\cong\to 
 M^I/U$, applying the functor $F_*$, we obtain
$F_*(M^I/U) \cong F_*(\varinjlim_{J \in U} M^J) \cong \varinjlim_{J \in U} F^*(M^J) \to  \varinjlim_{J \in U} (F^*(M))^J \cong (F_*(M))^I/U  $. 

Keeping track,  we obtain that the above morphism $t :  F_*(M^I/U) \to  (F_*(M))^I/U  $ establishes a comparison between
$F_*(\delta^M_{(I,U)}) : F_*(M) \to F_*(M^I/U)$ and  $\delta^{F_*(M)}_{(I,U)}) : F_*(M) \to F_*(M)^I/U$
$$ \delta^{F_*(M)}_{(I,U)}) = t \circ F_*(\delta^M_{(I,U)}) .$$

Since $ F_*(\delta^M_{(I,U)}) = F_*(g) \circ F_*(j)$, combining the equations we obtain
$$ \delta^{F_*(M)}_{(I,U)}) = t \circ F_*(g) \circ F_*(j).$$

Applying again the Fact, we conclude that $F_*(j) : F_*(M) \to F_*(M')$ is a pure morphism of Igr's.
\end{proof}

\begin{cor}
For each $n \in \mathbb{N}$, the functor  $F_n : \mathcal{A} \to  p\mathbb{F}_2-mod$ preserves pure embeddings. More precisely, if  $M, M' \in \mathcal{A}$ and  $j : M \to M'$ is a pure $L$-embedding, then $F_n(j) : F_n(M) \to k_n(M')$ is a pure morphism of pointed $\mathbb{F}_2$-modules (described in the first-order single sorted language adequate). In particular $F_n(j) : F_n(M) \to F_n(M')$ is an injective morphism of pointed $\mathbb{F}_2$-modules.
\end{cor}

\begin{cor}
Let  $M, M' \in \mathcal{A}$ and  $j : M \to M'$ is a pure $L$-embedding. If $MC(F_*(M'))$, then $MC(F_*(M))$. 
\end{cor}

\begin{proof}
This follows directly from the previous Corollary.  Indeed, suppose that holds $MC(F_*(M'))$. Since  $h'_n : F_n(M') \to F_{n+1}(M')$ and $F_n(j) : F_n(M) \to F_n(M')$ are injective morphisms, then, by a diagram chase,   $h_n : F_n(M) \to F_{n+1}(M)$ is an injective morphism too, thus holds $MC(F_*(M))$.
$$\xymatrix@!=6pc{F_nM\ar[d]_{F_n(j)}\ar[r]^{h_n} & F_{n+1}M\ar[d]^{F_{n+1}(j)} \\ F_n(M')\ar[r]_{h'_n} & F_{n+1}(M')}$$ \end{proof}

\section*{Appendix: Some Categorical Facts}

For the reader's convenience, we provide here some categorical results concerning adjunctions. Most of them are based on \cite{borceux1994handbook1}, but the reader could also consult \cite{mac2013categories}.

\begin{defn}[3.1.1 of \cite{borceux1994handbook1}]\label{3.1.1borceux}
Let $F:\mathcal A\rightarrow\mathcal B$ be a functor and $B$ an object of $\mathcal B$. A \textbf{reflection} of $B$ along $F$ is a pair $(R_B,\eta_B)$ where
\begin{enumerate}
    \item $R_B$ is an object of $A$ and $\eta_B:B\rightarrow F(R_B)$ is a morphism of $\mathcal B$.
    \item If $A\in\mathcal A$ is another object and $b:B\rightarrow F(A)$ is a morphism of $\mathcal B$, there exists a unique morphism $a:R_B\rightarrow A$ in $\mathcal A$ such that $F(a)\circ\eta_B=b$.
\end{enumerate}
\end{defn}

\begin{prop}[3.1.2 of \cite{borceux1994handbook1}]\label{3.1.2borceux}
Let $F:\mathcal A\rightarrow\mathcal B$ be a functor and $B$ an object of $\mathcal B$. When the reflection of $B$ along $F$ exists, it is unique up to isomorphism.
\end{prop}

\begin{defn}[3.1.4 of \cite{borceux1994handbook1}]\label{3.1.4borceux}
A functor $R:\mathcal B\rightarrow A$ is \textbf{left adjoint} to the functor $F:\mathcal A\rightarrow\mathcal B$ when there exists a natural transformation
$$\eta:1_{\mathcal B}\Rightarrow F\circ R$$
such that for every $B\in\mathcal B$, $(R(B),\eta_B)$ is a reflection of $B$ along $F$.
\end{defn}

\begin{teo}[3.1.5 of \cite{borceux1994handbook1}]\label{3.1.5borceux}
 Consider two functors $F:\mathcal A\rightarrow\mathcal B$ and $G:\mathcal B\rightarrow\mathcal A$. The following conditions are equivalent.
 \begin{enumerate}
     \item $G$ is left adjoint of $F$.
     \item There exist a natural transformation $\eta:1_{\mathcal B}\Rightarrow F\circ G$ and $\varepsilon:  G\rightarrow F\Rightarrow1_{\mathcal A}$ such that
     $$(F\ast\varepsilon)\circ(\eta\ast F)=1_F,\,(\varepsilon\ast G)\circ(G\ast\eta)=1_G.$$
     \item There exist bijections
     $$\theta_{AB}:\mathcal A(G(B),A)\cong\mathcal B(B,F(A))$$
     for every objects $A$ and $B$, and those bijections are natural both in $A$ and $B$.
     \item $F$ is right adjoint of $G$.
 \end{enumerate}
\end{teo}

\begin{prop}[3.2.2 of \cite{borceux1994handbook1}]\label{3.2.2borceux}
If the functor $F:\mathcal A\rightarrow\mathcal B$ has a left adjoint then $F$ preserves all limits which turn out to exist in $\mathcal A$.
\end{prop}

\begin{prop}[3.4.1 of \cite{borceux1994handbook1}]\label{3.4.1borceux}
Consider two functors $F:\mathcal A\rightarrow\mathcal B$, $G:\mathcal B\rightarrow\mathcal A$ with $G$ left adjoint to $F$ with $\eta:1_{\mathcal B}\Rightarrow F\circ G$ and $\varepsilon:  G\circ F\Rightarrow1_{\mathcal A}$ the two corresponding natural transformations. The following conditions are equivalent.
\begin{enumerate}
    \item $F$ is full and faithfull.
    \item $\varepsilon$ is an isomorphism.
\end{enumerate}
Under these conditions, $\eta\ast F$ and $G\ast\eta$ are isomorphisms as well.
\end{prop}

\begin{prop}[3.4.3 of \cite{borceux1994handbook1}]\label{3.4.3borceux}
Given a functor $F:\mathcal A\rightarrow\mathcal B$, the following conditions are equivalent:
\begin{enumerate}
    \item $F$ is full and faithfull and has a full and faithfull left adjoint $G$.
    \item $F$ has a left adjoint $G$ and the two canonical natural transformations of the adjunction $\eta:1_{\mathcal B}\Rightarrow F\circ G$ and $\varepsilon:  G\rightarrow F\Rightarrow1_{\mathcal A}$ are isomorphisms.
    \item There exists a functor $G:\mathcal B\rightarrow\mathcal A$ and two arbitrary natural isomorphisms $1_{\mathcal B}\cong F\circ G$, $G\circ F\cong1_{\mathcal A}$.
    \item $F$ is full and faitfull and each object $B\in\mathcal B$ is isomorphic to an object of the form $F(A)$, for some $A\in\mathcal A$.
    \item The dual condition of (1).
    \item The dual condition of (2).
\end{enumerate}
\end{prop}

\begin{defn}[3.4.4 of \cite{borceux1994handbook1}]\label{3.4.4borceux}
The categories $\mathcal A,\mathcal B$ are \textbf{equivalent} if there exist a functor $F:\mathcal A\rightarrow\mathcal B$ satisfying the conditions of Proposition \ref{3.4.3borceux}.
\end{defn}

\bibliographystyle{alpha}
\bibliography{one_for_all}

\end{document}